\documentclass[11pt]{amsart}
\usepackage{amssymb}
\usepackage{amsmath}
\usepackage{amsfonts}
\usepackage{amsxtra}
\usepackage[mathscr]{eucal}
\usepackage{color}

\theoremstyle{plain}
\newtheorem{theorem}{Theorem}[section]
\newtheorem{proposition}[theorem]{Proposition}
\newtheorem{lemma}[theorem]{Lemma}

\theoremstyle{remark}
\newtheorem{remark}[theorem]{Remark}
\newtheorem*{notation}{Basic notation and terminology}

\newcommand{\reft}[1]{Theorem \ref{thm:#1}}
\newcommand{\refp}[1]{Proposition \ref{prop:#1}}

\newcommand{\refl}[1]{Lemma \ref{lem:#1}}
\newcommand{\refeq}[1]{Eq. \eqref{eq:#1}}
\newcommand{\refeqs}[2]{Eqs. \eqref{eq:#1} and \eqref{eq:#2}}
\newcommand{\refr}[1]{Remark \ref{rmk:#1}}

\newcommand{\refs}[1]{Section \ref{sec:#1}}

\numberwithin{equation}{section}

\def\alp{\alpha}
\def\eps{\varepsilon}
\def\lam{\lambda}
\def\ome{\omega}
\def\sig{\sigma}
\def\vsig{\varsigma}
\def\vphi{\varphi}
\def\vrho{\varrho}
\def\tet{\vartheta}
\def\zet{\zeta}
\def\Ome{\Omega}
\def\VPhi{\varPhi}
\def\VPsi{\varPsi}
\def\oalp{\overline\alpha}

\def\hchi{\widehat{\chi}}
\def\hxi{\widehat{\xi}}
\def\hzet{\widehat{\zet}}
\def\hphi{\widehat{\phi}}
\def\htet{\widehat{\vartheta}}
\def\hsig{\widehat{\varsigma}}
\def\htau{\widehat{\tau}}
\def\D{\mathfrak{D}}
\def\CA{\mathscr{A}}

\def\CC{\mathscr{C}}
\def\CD{\mathscr{D}}

\def\CI{\mathscr{I}}
\def\CJ{\mathscr{J}}
\def\CK{\mathscr{K}}
\def\CL{\mathscr{L}}
\def\CM{\mathscr{M}}

\def\CX{\mathscr{X}}
\def\CY{\mathscr{Y}}
\def\HCC{\hat{\mathscr{C}}}
\def\HCK{\widehat{\mathscr{K}}}
\def\ed{e_{\CD}(\vphi)}
\def\KD{\mathscr{K}_{\CD}(\vphi)}
\def\HKD{\widehat{\mathscr{K}}_{\CD}(\vphi)}
\def\xd{\xi_{\CD}(\vphi)}
\def\hxd{\hxi_{\CD}(\vphi)}
\def\ld{\lam_{\CD}(\vphi)}

\def\hthd{\htet_{\CD}(\vphi)}
\def\sd{\vsig_{\CD}(\vphi)}
\def\td{\tau_{\CD}(\vphi)}
\def\htd{\htau_{\CD}(\vphi)}

\def\fru{\mathfrak{ut}}
\def\k{\Bbbk}
\def\C{\mathbb{C}}
\def\cx{\mathbb{C}^{\x}}
\def\kx{\k^{\x}}

\def\ax{\CA^{\times}}

\def\ks{\k^{\sigma}}

\def\ksx{(\k^{\sigma})^{\times}}
\def\ic{\CI^{\circ}}
\def\icc{C_{\CI}(\sig)^{\circ}}
\def\jc{\CJ^{\circ}}
\def\jcc{C_{\CJ}(\sig)^{\circ}}

\def\ds{\CD^{\sigma}}
\def\phs{\vphi^{\sigma}}
\def\sset{\subseteq}
\newcommand{\map}[3]{#1 \colon #2 \to #3}
\newcommand{\frob}[2]{\langle #1 , #2 \rangle}
\newcommand{\set}[2]{\{ #1 \colon #2 \}}
\newcommand{\bset}[2]{\big\{ #1 \colon #2 \big\}}
\newcommand{\ger}[1]{\langle #1 \rangle}
\newcommand{\seq}[2]{#1_{1}, \ldots, #1_{#2}}
\def\x{\times}
\def\sn{[[n]]}
\def\inv{^{-1}}
\def\trp{^{\mathtt{T}}}
\def\ort{^{\perp}}
\def\ovl{\overline}
\def\irr{\operatorname{Irr}}
\def\sch{\operatorname{SCh}}
\def\cf{\operatorname{cf}}
\def\scf{\operatorname{scf}}
\def\scl{\operatorname{SCl}}
\def\GL{\operatorname{GL}}
\def\UT{\operatorname{UT}}
\def\fr{\operatorname{Fr}}
\def\and{\quad \text{and} \quad}

\allowdisplaybreaks

\begin{document}


\title[]{A supercharacter theory for involutive algebra groups}

\author[]{Carlos A. M. Andr\'e, Pedro J. Freitas \and Ana Margarida Neto}

\address[C. A. M. Andr\'e \& P. J. Freitas]{Departamento de Matem\'atica \\ Faculdade de Ci\^encias da Uni\-ver\-si\-da\-de de Lis\-boa \\ Cam\-po Grande \\ Edi\-f\'\i \-cio C6 \\ Piso 2 \\ 1749-016 Lisboa \\ Portugal}

\address[A. M. Neto]{Instituto Superior de Economia e Gest\~ao, Universidade de Lisboa, Rua do Quelhas 6, 1200-781 Lisboa, Portugal}

\address[C. A. M. Andr\'e, P. J. Freitas \& A. M. Neto]{Centro de Estruturas Lineares e Combinat\'orias, Instituto Interdisciplicar da Universidade de Lisboa, Av. Prof. Gama Pinto 2, 1649-003 Lisboa, Portugal}

\email{caandre@fc.ul.pt}

\email{pedro@ptmat.fc.ul.pt}

\email{ananeto@iseg.utl.pt}


\date{\today}


\thanks{This research was made within the activities of the Centro de Estruturas Lineares e Combinat\'orias and was partially supported by the Funda\c c\~ao para a Ci\^encia e Tecnologia (Portugal).}

\begin{abstract}
If $\CJ$ is a finite-dimensional nilpotent algebra over a finite field $\k$, the algebra group $P = 1+\CJ$ admits a (standard) supercharacter theory as defined in \cite{DiaIsa}. If $\CJ$ is endowed with an involution $\sig$, then $\sig$ naturally defines a group automorphism of $P = 1+\CJ$, and we may consider the fixed point subgroup $C_{P}(\sig) = \set{x\in P}{\sig(x) = x\inv}$. Assuming that $\k$ has odd characteristic $p$, we use the standard supercharacter theory for $P$ to construct a supercharacter theory for $C_{P}(\sig)$. In particular, we obtain a supercharacter theory for the Sylow $p$-subgroups of the finite classical groups of Lie type, and thus extend in a uniform way the construction given by Andr\'e and Neto in \cite{AndNet1,AndNet2} for the special case of the symplectic and orthogonal groups.
\end{abstract}

\maketitle


\section{Introduction} \label{sec:intro}

The notion of a supercharacter theory of a finite group was introduced by P. Diaconis and I.M. Isaacs in \cite{DiaIsa} to generalise the \textit{basic characters} defined by C. Andr\'e in \cite{And1, And2, And3}, and the \textit{transition characters} defined by N. Yan in his PhD thesis \cite{Yan1} (see also \cite{Yan2}). Both basic and transition characters were introduced with the aim of approaching the usual character theory of the finite group $\UT_{n}(\k )$ consisting of $n \x n$ unimodular upper-triangular matrices over a finite field $\k$ of characteristic $p$. (By ``unimodular'', we mean that all diagonal entries are equal to $1$; we will refer to $\UT_{n}(\k)$ simply as a (finite) \textit{unitriangular group}.) The basic idea is to coarsen the usual character theory of a group by replacing irreducible characters with linear combinations of irreducible characters that are constant on a set of clumped conjugacy classes.

Let $G$ be a finite group, and write $\irr(G)$ to denote the set of irreducible characters of $G$. (Throughout the paper, all characters are taken over the field $\C$ of complex numbers.) Let $\CK$ be a partition of $G$, and let $\CX$ be a partition of $\irr(G)$. (Here, and throughout this paper, when we use the word ``partition'', we require that the parts are all non-empty.) For each $X \in \CX$, we define
\begin{equation} \label{eq:sig1}
\sig_{X} = \sum_{\psi \in X} \psi(1) \psi,
\end{equation}
and note that $\sum_{X \in \CX} \sig_{X} = \rho_{G}$, the regular character of $G$. (Recall that $\rho_{G}(g) = 0$ for all $g \in G$, $g \neq 1$, and $\rho_{G}(1) = |G|$.) We recall from \cite{DiaIsa} that the pair $(\CX,\CK)$ is called a \textit{supercharacter theory} for $G$ provided that the following conditions hold.
\begin{enumerate}
\item[(S1)] $|\CX| = |\CK|$.
\item[(S2)] $\{1\} \in \CK$.
\item[(S3)] For each $X \in \CX$, the character $\sig_{X}$ is constant on each member of $\CK$.
\end{enumerate}
As shown in \cite[Lemma~2.1]{DiaIsa} this definition is equivalent to the following (see \cite{AndNet3}). A \textit{supercharacter theory} for a finite group $G$ is a pair $(\CX,\CK)$ where $\CK$ is a partition of $G$, $\CX$ is a collection of characters og $G$, and the following conditions hold.
\begin{enumerate}
\item[(S1')] $|\CX| = |\CK|$.
\item[(S2')] Every irreducible character of $G$ is a constituent of a unique $\chi \in \CX$.
\item[(S3')] Every $\chi \in \CX$ is constant on each member of $\CK$.
\end{enumerate}
We refer to the elements of $\CX$ as the {\it supercharacters} of $G$, and to each $K \in \CK$ as a {\it superclass} of $G$. Regardless of which definition one chooses to work with, it is straightforward to verify that each superclass is a union of conjugacy classes of $G$ and that each of the partitions $\CK$ and $\CX$ determines the other. The only significant difference between these two definitions is that the second approach can yield supercharacters which are multiples of the characters $\sig_{X}$ defined above.

In the literature to date, one of the main uses of supercharacter theory has been to perform computations when a complete character theory is difficult or impossible to determine. For instance, an explicit computation of the irreducible characters and the conjugacy classes of the finite unitriangular groups $\UT_{n}(\k)$ is known to be a ``wild'' problem, but Andr\'e \cite{And1} and Yan \cite{Yan1} have developed an applicable supercharacter theory in this situation. (Andr\'e's original approach works only when the characteristic of $\k$ is large enough, although he extends this to the general case in the later paper \cite{And3}; Yan's construction is slighty different and much more elementary, and it yields the same supercharacter theory as Andr\'e's.) In \cite{DiaIsa}, Diaconis and Isaacs generalise Yan's approach in order to extend the supercharacter theory of $\UT_{n}(\k)$ to a much larger class of $p$-groups introduced by Isaacs in \cite{Isa2}, namely algebra groups over a finite field $\k$ of characteristic $p$. Let $\CA$ be a finite-dimensional associative $\k$-algebra (with identity), and write $\ax$ to denote the unit group of $\CA$ (that is, the group of invertible elements of $\CA$). Following the terminology of \cite{Isa2}, given any nilpotent subalgebra $\CJ$ of $\CA$, the {\it algebra group} based on $\CJ$ is the multiplicative subgroup $1+\CJ$ of $\ax$; notice that a subalgebra of $\CA$ is not required to contain the identity (it is simply a multiplicatively closed vector subspace of $\CA$). We note that $\k\cdot 1 + \CJ$ is a (local) subalgebra of $\CA$, and that $P = 1+\CJ$ is a (normal) Sylow $p$-subgroup of the unit group $(\k\cdot 1 + \CJ)^{\x}$; indeed, $(\k\cdot 1 + \CJ)^{\x}$ is isomorphic to the direct product $\kx \x P$. In fact, it is shown in \cite[Theorem~1.5]{And4} that a finite group is an algebra group over $\k$ if and only if it is a Sylow $p$-subgroup of the unit group of some finite-dimensional $\k$-algebra $\CA$. These algebra groups generalise the finite unitriangular groups over $\k$; in this standard example, we let $\CA = \CM_{n}(\k)$ be the $\k$-algebra consisting of all $n \x n$ matrices with entries in $\k$, so that $\ax = \GL_{n}(\k)$ is the general linear group consisting of all invertible matrices in $\CM_{n}(\k)$. Then, $\UT_{n}(\k) = 1+\CJ$ is the algebra group based on the nilpotent subalgebra $\CJ = \fru_{n}(\k)$ of $\CM_{n}(\k)$ which consists of all strictly upper-triangular matrices.

The primary aim of this paper is to develop a supercharacter theory for another family of $p$-groups which are associated with finite-dimensional nilpotent $\k$-algebras with involution. These $p$-groups include the Sylow $p$-subgroups of the finite classical groups of Lie type, and our construction is motivated by the methods used by C. Andr\'e and A.M. Neto in \cite{AndNet1, AndNet2, AndNet3} for the particular case of the Sylow $p$-subgroups of the symplectic group $Sp_{2m}(\k)$, and the orthogonal groups $O^{+}_{2m}(\k)$ and $O_{2m+1}(\k)$ (see below). We assume that $\k$ is a finite field of odd characteristic $p$, and let $\CA$ is a finite-dimensional $\k$-algebra endowed with an involution. We recall that an {\it involution} on $\CA$ is a map $\map{\sig}{\CA}{\CA}$ satisfying the following conditions:
\begin{enumerate}
\item $\sig(a+b) = \sig(a)+\sig(b)$ for all $a,b \in \CA$;
\item $\sig(ab) = \sig(b)\sig(a)$ for all $a,b \in \CA$;
\item $\sig^{2}(a) = a$ for all $a \in \CA$.
\end{enumerate}
We note that an involution $\sig$ is not required to be $\k$-linear; however, we will assume that the field $\k = \k\cdot 1$ is preserved by $\sig$ \footnote{This essential assumption is missing in the definition given in \cite{And4}; however, it is implicit throughout that paper and all results are valid under this hypothesis. The first author is grateful to I.M. Isaacs for pointing this out to him.}. Then, $\sig$ defines a field automorphism of $\k$ which is either the identity or has order $2$; we say that $\sig$ is {\it of the first kind} if $\sig$ fixes $\k$, and {\it of the second kind} if its restriction $\sig_{\k}$ to $\k$ has order $2$. In any case, we let $\ks = \set{\alp \in \k}{\sig(\alp) = \alp}$ denote the $\sig$-fixed subfield of $\k$, and consider that $\CA$ is a finite dimensional associative $\ks$-algebra. We observe that $\sig$ is of the second kind if and only if the field extension $\ks \sset \k$ has degree $2$, and $\map{\sig}{\k}{\k}$ is the {\it Frobenius map} defined by the mapping $\alp \mapsto \alp^{q}$ where $q = |\ks|$; for simplicity of writing, we will use the bar notation $\oalp = \alp^{q}$ for $\alp \in \k$.

An important example occurs in the case where $\CA = \CM_{n}(F)$ is endowed with the canonical {\it transpose involution} given by the mapping $a \mapsto a\trp$ where $a\trp$ denotes the transpose of $a \in \CM_{n}(F)$. More generally, let $q = |\ks|$, let $\map{\fr_{q}}{\CM_{n}(\k)}{\CM_{n}(\k)}$ be the Frobenius morphism defined by $\fr_{q}(a_{ij}) = (\ovl{a}_{ij}) = ({a_{ij}}^{q})$ for all $(a_{ij}) \in \CM_{n}(\k)$, and set $a^{\ast} = \fr_{q}(a)\trp$ for all $a \in \CM_{n}(\k)$. Then, the mapping $a \mapsto a^{\ast}$ defines an involution on $\CM_{n}(\k)$; notice that, if $\ks = \k$, then $a^{\ast} = a\trp$ for all $a \in \CM_{n}(\k)$. If $\map{\sig}{\CM_{n}(\k)}{\CM_{n}(\k)}$ is an involution of the first kind, then there exists $u \in \GL_{n}(\k)$ with $u\trp = \pm u$ and such that $\sig(a) = u\inv a\trp u$ for all $a \in \CM_{n}(\k)$; moreover, the matrix $u$ is uniquely determined up to a factor in $\kx$. On the other hand, if $\map{\sig}{\CM_{n}(\k)}{\CM_{n}(\k)}$ is an involution of the second kind, then there exists $u \in \GL_{n}(\k)$ with $u^{\ast} = u$ and such that $\sig(a) = u\inv a^{\ast} u$ for all $a \in \CM_{n}(\k)$; moreover, the matrix $u$ is uniquely determined up to a factor in $\ksx$. [The proofs can be found in the book \cite{KMRT} by M.-A. Knus {\it et al.} (see, in particular, Propositions~2.19 and 2.20) where the complete classification of involutions is also given for arbitrary central $\k$-algebras (see Propositions 2.7 and 2.18).] For simplicity, for $u \in \GL_{n}(\k)$ as above, we will denote by $\sig_{u}$ the involution on $\CM_{n}(F)$ given by the mapping $a \mapsto u\inv a^{\ast} u$; as usual, we say that $\sig_{u}$ is {\it symplectic} if $\sig_{u}$ is of the first kind and $u\trp = -u$, {\it orthogonal} if $\sig_{u}$ is of the first kind and $u\trp = u$, and {\it unitary} if $\sig_{u}$ is of the second kind and $u^{\ast} = u$.

In the general situation, consider the unit group $\ax$ of the $\k$-algebra $\CA$. Then, for any involution $\map{\sig}{\CA}{\CA}$, the cyclic group $\ger{\sig}$ acts on $\ax$ as a group of automorphisms by means of $x^{\sig} = \sig(x\inv)$ for all $x \in \ax$ ($x^{\sig}$ should not be confused with $\sig(x)$). For any $\sig$-invariant subgroup $H $ of $\ax$, we denote by $C_{H}(\sig)$ the subgroup of $H$ consisting of all $\sig$-fixed elements; that is, $C_{H}(\sig) = \set{x \in H}{x^{\sig} = x} = \set{x \in H}{\sig(x\inv) = x}$. In the case where $\CA = \CM_{n}(\k)$, an arbitrary involution $\map{\sig}{\CM_{n}(\k)}{\CM_{n}(\k)}$ defines a group $C_{\GL_{n}(\k)}(\sig)$ which is isomorphic to one of the {\it finite classical groups of Lie type} (defined over $\k$): the {\it symplectic group} $Sp_{2m}(\k)$ if $\sig$ is symplectic, the {\it orthogonal groups} $O^{+}_{2m}(\k)$, $O_{2m+1}(\k)$, or $O^{-}_{2m+2}(\k)$ if $\sig$ is orthogonal, and the {\it unitary group} $U_{n}(\k)$ if $\sig$ is unitary. [For the details on the definition of the classical groups, we refer to Chapter I the book \cite{Car1} by R. Carter.] In fact, up to isomorphism, these groups may be defined by the involution $\sig = \sig_{u}$ where $u \in \GL_{n}(\k)$ is the matrix defined as follows; here, $J_{m}$ denotes the $m \x m$  matrix with $1$'s along the anti-diagonal and $0$'s elsewhere.
\begin{enumerate}
\item For $Sp_{2m}(\k)$, we choose $u = \left( \begin{smallmatrix} 0 & J_{m} \\ -J_{m} & 0 \end{smallmatrix} \right)$.
\item For $O^{+}_{2m}(\k)$ or $O_{2m+1}(\k)$, we choose $u = J_{n}$ where, either $n = 2m$, or $n = 2m+1$.
\item For $O^{-}_{2m+2}(\k)$, we choose $u = \left( \begin{smallmatrix} 0 & 0 & J_{m} \\ 0 & c & 0 \\ J_{m} & 0 & 0 \end{smallmatrix} \right)$ where $c = \left( \begin{smallmatrix} 1 & 0 \\ 0 & -\eps \end{smallmatrix} \right)$ for $\eps \in \kx \setminus (\kx)^{2}$.
\item For $U_{n}(\k)$, we choose $u = J_{n}$.
\end{enumerate}
We refer to $\sig = \sig_{u}$ (for this matrix $u$) as a \textit{canonical involution} on $\CM_{n}(\k)$.

As we mentioned above, our main goal in this paper is to develop a supercharacter theory for the group $C_{P}(\sig)$ in the case where $P$ is a $\sig$-invariant algebra subgroup of $\ax$. Our construction is given in terms of the supercharacter theory of $P$, and extends the results of \cite{AndNet1, AndNet2, AndNet3} in the particular case where $P = \UT_{n}(\k)$ is the unitriangular group over $\k$ and $C_{P}(\sig)$ is the Sylow $p$-subgroup of $Sp_{2m}(\k)$, $O^{+}_{2m}(\k)$ or $O_{2m+1}(\k)$. More generally, our construction applies to the particular case where $\CA = \CM_{n}(\k)$, and $\map{\sig}{\CM_{n}(\k)}{\CM_{n}(\k)}$ is any canonical involution. In this situation, it is well-known that the Sylow $p$-subgroups of $C_{\GL_{n}(\k)}(\sig)$ are conjugate to the $\sig$-fixed subgroup $C_{P}(\sig)$ where $P$ is, either is the unitriangular subgroup $\UT_{n}(\k)$ of $\GL_{n}(\k)$, or the subgroup of $\UT_{n}(\k)$ consisting of all unimodular upper-triangular matrices with $(m+1,m+2)$th position equal to zero. The former situation occurs only if $G$ is the orthogonal group $O^{-}_{2m+2}(q)$; indeed, the unitriangular group is not invariant for the corresponding involution. [In this case, the supercharacter theory of $P$ has a slighty different parametrization than that of $\UT_{n}(F)$, and thus the supercharacter theory of $C_{P}(\sig)$ has to be described separately; we leave this description as an exercise for the reader.]

To conclude this introduction, we mention that supercharacter theories have proven to be relevant outside the realm of finite group theory. For instance, as shown in \cite{DiaIsa} these notions can be used to obtain a more general theory of spherical functions and Gelfand pairs. Another application may be found in \cite{ArDiaSta} where the supercharacter theory of $\UT_{n}(\k)$ is applied to study random walks on upper-triangular matrices. In a different direction, recent work has revealed deep connections between the supercharacter theory of $\UT_{n}(\k)$ and the Hopf algebra of symmetric functions of noncommuting variables (see \cite{AIM,BJBerTh,BerTh}). We hope that analogous applications and connections could be derived using the supercharacter theories developed in this paper (see the recent paper \cite{Ben} by C. Benedetti). Finally, we also mention the relation between supercharacter theories and Schur rings discovered by O. Hendrickson in \cite{Hen}, and the applications of supercharacter theories of finite abelian groups to exponential sums in number theory (see \cite{FleGarKar,FowGarKar}).

\begin{notation}
Throughout the paper, we let $\k$ denote a finite field with odd characteristic $p$, let $\CA$ be a finite-dimensional $\k$-algebra endowed with an involution $\map{\sig}{\CA}{\CA}$, and let $\CJ$ be a $\sig$-invariant nilpotent subalgebra of $\CA$. Let $\ax$ denote the unit group of $\CA$, and let $P = 1+\CJ$ be the algebra subgroup of $\ax$ based on $\CJ$. Then, $P$ is $\sig$-invariant with respect to the action given by
\begin{equation} \label{eq:action1}
x^{\sig} = \sig(x\inv)
\end{equation}
for all $x \in \ax$. As usual, we write $C_{P}(\sig)$ to denote the subgroup of $P$ consisting of all $\sig$-fixed elements, that is, $$C_{P}(\sig) = \set{x\in P}{x^{\sig} = x}.$$

We define the {\it Cayley transform} $\map{\VPhi}{\CJ}{P}$ by the rule
\begin{equation} \label{eq:cayley1}
\VPhi(a) = (1+a)(1-a)\inv = 1+2a(1-a)\inv
\end{equation}
for all $a \in \CJ$; notice that $(1-a)\inv = 1 + a(1-a)\inv$ for all $a \in \CJ$. Since $p$ is odd, this map is bijective, and its inverse $\map{\VPsi}{P}{\CJ}$ is given by
\begin{equation} \label{eq:cayley2}
\VPsi(x)=(x-1)(x+1)^{-1}
\end{equation}
for all $x \in P$. It is clear that $\VPhi(\sig(a)) = \sig(\VPhi(a))$ for all $a \in \CJ$, and so the Cayley transform restricts to a bijective map $\map{\VPhi}{C_{\CJ}(\sig)}{C_{P}(\sig))}$ where we set $$C_{\CJ}(\sig) = \set{a \in \CJ}{\sig(a) = -a};$$ notice that $C_{\CJ}(\sig)$ is a vector space over the $\sig$-fixed subfield $\ks$ of $\k$. Throughout the paper, we consider the action of $\sig$ on $\CJ$ defined by
\begin{equation} \label{eq:action2}
a^{\sig} = -\sig(a)
\end{equation}
for all $a \in \CJ$, so that $C_{\CJ}(\sig) = \set{a \in \CJ}{a^{\sig} = a}$ is the (additive) subgroup of $\CJ$ consisting of all $\sig$-fixed elements. We observe that this action commutes with $\VPhi$, that is,
\begin{equation} \label{eq:commute}
\VPhi(a^{\sig}) = \VPhi(a)^{\sig}
\end{equation}
for all $a \in \CJ$; notice also that $\VPsi(x^{\sig}) = \VPsi(x)^{\sig}$ for all $x \in P$.

On the other hand, we denote by $\jc$ the \textit{dual group} of $\CJ^{+}$ which by definition consists of all linear characters $\map{\lam}{\CJ^{+}}{\C}$ of the additive group $\CJ^{+}$ of $\CJ$; since $\CJ^{+}$ is an abelian group, it is a standard fact that $\jc$ is the set $\irr(\CJ^{+})$ of all irreducible characters of $\CJ^{+}$. We note that $\jc$ is an abelian group with respect to the product of characters defined by $(\lam\mu)(a) = \lam(a)\mu(a)$ for all $\lam, \mu \in \jc$ and all $a \in \CJ$; in particular, notice that $\lam^{2}(a) = \lam(a)\lam(a) = \lam(2a)$ for all $\lam \in \jc$ and all $a \in \CJ$. For every $\lam \in \jc$, we define the linear character $\lam^{\sig} \in \jc$ by
\begin{equation} \label{eq:action3}
\lam^{\sig}(a) = \lam(a^{\sig}) = \lam(-\sig(a))
\end{equation}
for all $a \in \CJ$. This clearly defines an action of $\sig$ on $\jc$, and thus we can define the $\sig$-fixed subgroup $C_{\jc}(\sig) = \set{\lam \in \jc}{\lam^{\sig} = \lam}$ of $\jc$. However, we prefer to realise this subgroup as the dual group $\jcc$ of the additive group $C_{\CJ}(\sig)^{+}$ of $C_{\CJ}(\sig)$. In fact, it is easily seen that $\CJ$ decomposes as the direct sum $\CJ = C_{\CJ}(\sig) \oplus [\CJ,\sig]$ where $[\CJ,\sig] = \set{a+\sig(a)}{a \in \CJ}$, and thus $\jcc$ can be naturally identified with the orthogonal subgroup $[\CJ,\sig]\ort$; for any additive subgroup $\CI$ of $\CJ$, the \textit{orthogonal subgroup} $\CI\ort$ is defined by $\CI\ort = \set{\lam \in \jc}{\CI \sset \ker(\lam)}$. In light of the above identification, we see that
\begin{equation} \label{eq:jcc}
\jcc = \set{\lam \in \jc}{\lam^{\sig} = \lam};
\end{equation}
indeed, the \refeq{action3} implies that for every $\lam \in \jc$ we have $\lam^{\sig} = \lam$ if and only if $\lam(a+\sig(a)) = 1$ for all $a \in \CJ$.

\end{notation}


\section{Superclasses} \label{sec:super1}

Let $\CJ$ be a $\sig$-invariant nilpotent subalgebra of $\CA$, and let $P = 1+\CJ$. Then, right multiplication defines a right action of $P$ on $\CJ$, whereas left multiplication defines a left action of $P$ on $\CJ$; these two actions are compatible in the sense that $(xa)y = x(ay)$ for all $x,y \in P$ and all $a \in \CJ$. It follows that $\CJ$ decomposes as a disjoint union of \textit{two-sided orbits} $PaP$ for $a \in \CJ$. Then, the \textit{superclasses} of the algebra group $P$ are defined be the subsets of the form $1+PaP$ where $a \in \CJ$; we write $\scl(P)$ to denote the set of all superclasses of $P$. We note that, for any $a \in \CJ$, the set $PaP$ is an orbit for the natural action of $P\x P$ on $\CJ$ given by $(x,y) \cdot a = xay\inv$ for all $a \in \CJ$ and all $x,y \in P$, and that every superclass is a (disjoint) union of conjugacy classes. In fact, every two-sided orbit on $\CJ$ is a disjoint union of \textit{conjugation orbits} where the \textit{conjugation action} $P\x \CJ \to \CJ$ is defined by the mapping $(x,a) \mapsto xax\inv$.

The purpose of this section is to define superclasses of the $\sig$-fixed subgroup $C_{P}(\sig)$ of $P$, and the most natural way of defining them is to consider the non-empty intersections $\HCK \cap C_{P}(\sig)$ where $\HCK$ is a superclass of $P$. [Throughout the paper, we shall use the hat notation $\HCK$ for superclasses of $P$, and reserve the notation $\CK$ for the intersection $\HCK \cap C_{P}(\sig)$.] Obviously, the intersection $\HCK \cap C_{P}(\sig)$ is non-empty if and only if $\HCK$ contains an element which is fixed by $\sig$, and we shall prove that this is equivalent to requiring that the superclass $\HCK$ is $\sig$-invariant. We start by proving an alternative description of the superclasses of $P$ in terms of the Cayley transform $\map{\VPhi}{\CJ}{P}$; this is crucial for our work because $\VPhi$ defines a bijection from $C_{\CJ}(\sig)$ to $C_{P}(\sig)$ (whereas the standard mapping $a \mapsto 1+a$ does not).

\begin{lemma} \label{lem:scl1}
Let $\CJ$ be a $\sig$-invariant nilpotent subalgebra of $\CA$, and let $P = 1+\CJ$. If $a \in \CJ$ and $\HCK \in \scl(P)$ is the superclass which contains $x = \VPhi(a)$, then $\HCK = 1+P(2a)P = \VPhi(PaP)$. In particular, $\scl(P) = \set{\VPhi(PaP)}{a \in \CJ}$.
\end{lemma}

\begin{proof}
Since $x = \VPhi(a) = 1+2a(1-a)\inv$, we clearly have $x \in 1+P(2a)P$, and thus $\HCK = 1+P(2a)P$. If $y,z \in P$, then $\VPhi(yaz) \in 1+P(2yaz)P = 1+P(2a)P$, and thus $\VPhi(PaP) \sset 1+P(2a)P$. The result follows because $\VPhi$ is bijective and $|PaP| = |P(2a)P|$.
\end{proof}

Next, we observe that the cyclic group $\ger{\sig}$ acts on the set $\scl(P)$.

\begin{lemma} \label{lem:scl2}
Let $\CJ$ be a $\sig$-invariant nilpotent subalgebra of $\CA$, and let $P = 1+\CJ$. If $\HCK \in \scl(P)$, then $\HCK^{\sig} \in \scl(P)$; in fact, if $\HCK = \VPhi(PaP)$ for $a \in \CJ$, then $\HCK^{\sig} = \VPhi(Pa^{\sig}P)$.
\end{lemma}

\begin{proof}
It is enough to use \refeq{commute} since $(xay)^{\sig} = y^{-\sig}a^{\sig}x^{-\sig}$ for all $x,y \in P$; as usual, we write $z^{-\sig} = (z\inv)^{\sig}$ for all $z \in P$. 
\end{proof}

Henceforth, we denote by $\scl_{\sig}(P)$ the subset of $\scl(P)$ consisting of all $\sig$-invariant superclasses of $P$. By \cite[Corollary~13.10]{Isa1}, every conjugacy class $\CC$ of $C_{P}(\sig)$ is the intersection $\CC = \HCC \cap C_{P}(\sig)$ for some $\sig$-invariant conjugacy class $\HCC$ of $P$, and moreover the mapping $\HCC \mapsto \HCC \cap C_{P}(\sig)$ defines a bijection between the set of $\sig$-invariant conjugacy class of $P$ and the set of conjugacy classes of $C_{P}(\sig)$. Therefore, for every superclass $\HCK \in \scl(P)$, either the intersection $\HCK \cap C_{P}(\sig)$ is empty, or it is a union of conjugacy classes of $C_{P}(\sig)$; this is one of the conditions which should be satisfied by any set of superclasses. We define a \textit{superclass} of $C_{P}(\sig)$ to be a non-empty intersection $\HCK \cap C_{P}(\sig)$ for $\HCK \in \scl(P)$, and denote by $\scl(C_{P}(\sig))$ the set of all superclasses of $C_{P}(\sig)$. [Eventually, we will define the supercharacters of $C_{P}(\sig)$, and we will see that these definitions are compatible with the general definition of a supercharacter theory.] We have the following result.

\begin{proposition} \label{prop:scl1}
Let $\CJ$ be a $\sig$-invariant nilpotent subalgebra of $\CA$, let $P = 1+\CJ$, and let $\HCK \in \scl(P)$. Then, the intersection $\HCK \cap C_{P}(\sig)$ is non-empty if and only if the superclass $\HCK$ is $\sig$-invariant.
\end{proposition}

\begin{proof}
If $\HCK \cap C_{P}(\sig)$ is non-empty and $x \in \HCK \cap C_{P}(\sig)$, then $x \in \HCK \cap \HCK^{\sig}$. Since $\HCK^{\sig}$ is a superclass of $P$, it follows that $\HCK = \HCK^{\sig}$. Conversely, suppose that $\HCK = \HCK^{\sig}$, and let $a \in \CJ$ be such that $\VPhi(a) \in \HCK$. By the previous lemma, we have $PaP = Pa^{\sig}P$. Now, we consider the automorphism of the group $P\x P$ defined by the mapping $(x,y) \mapsto (x,y)^{\sig} = (y^{\sig},x^{\sig})$, and observe that $$((x,y)\cdot a)^{\sig} = (xay\inv)^{\sig} = y^{\sig}a^{\sig}x^{-\sig} = (y^{\sig},x^{\sig}) \cdot a^{\sig} = (x,y)^{\sig} \cdot a^{\sig}$$ for all $x,y \in P$ and all $a \in \CJ$. Thus, since $P\x P$ acts transitively on $PaP$ (and since $2 \nmid |P|$), Glauberman's Lemma (see \cite[Lemma~13.8]{Isa1}) implies that there exists $b \in PaP$ such that $b^{\sig} = b$. By \refeq{commute} and by the previous lemma, we conclude that the element $x = \VPhi(b) \in P$ satisfies $x^{\sig} = x$ and lies in $\HCK$.
\end{proof}

It follows that
\begin{equation} \label{eq:scl}
\scl(C_{P}(\sig)) = \set{\HCK \cap C_{P}(\sig)}{\HCK \in \scl_{\sig}(P)};
\end{equation}
moreover, the mapping $\HCK \mapsto \HCK \cap C_{P}(\sig)$ defines a bijection between $\scl_{\sig}(P)$ and $\scl(C_{P}(\sig))$. As we observed above, since every superclass of $P$ is a union of conjugacy classes, \cite[Corollary~13.10]{Isa1} implies that every superclass of $C_{P}(\sig)$ is also a union of conjugacy classes. Indeed, the following result also implies that every superclass of $C_{P}(\sig)$ is invariant under conjugation. 

\begin{theorem} \label{thm:scl2}
Let $\CJ$ be a $\sig$-invariant nilpotent subalgebra of $\CA$, and let $P = 1+\CJ$. If $a \in C_{\CJ}(\sig)$ and $\HCK \in \scl(P)$ contains $x = \VPhi(a) \in C_{P}(\sig)$, then $$\HCK \cap C_{P}(\sig) = \set{\VPhi(zaz^{-\sig})}{z \in P}$$ where we write $z^{-\sig} = (z\inv)^{\sig}$ for all $z \in P$.
\end{theorem}

\begin{proof}
By \refl{scl1} and \refeq{commute}, we see that $\HCK \cap C_{P}(\sig) = \set{\VPhi(u)}{u \in PaP,\ u^{\sig} = u}$. As in the proof of \refp{scl1}, we consider $\sig$ as the automorphism of $P\x P$ given by the mapping $(y,z) \mapsto (z^{\sig},y^{\sig})$. It follows by \cite[Corollary~13.9]{Isa1} that the set $\set{u \in PaP}{u^{\sig} = u}$ is an orbit for the action of the subgroup $C_{P\x P}(\sig) = \set{(z,z^{\sig})}{z \in P}$.  In other words, we have $\set{u \in PaP}{u^{\sig} = u} = \set{zuz^{-\sig}}{z \in P}$, and the result follows.
\end{proof}

We note that the algebra group $P$ acts on the left of $C_{\CJ}(\sig)$ by the rule $x \cdot a = x\inv ax^{\sig}$ for all $x \in P$ and all $a \in C_{\CJ}(\sig)$. Then, the previous theorem asserts that the superclass of $C_{P}(\sig)$ which contains an element $x \in C_{P}(\sig)$ is the image $\VPhi(\Ome_{P}(a))$ of the orbit $\Ome_{P}(a) = \set{x\inv ax^{\sig}}{x \in P}$ which contains the element $a \in C_{\CJ}(\sig)$ such that $x = \VPhi(a)$.

\section{Supercharacters} \label{sec:super2}

In this section we define the supercharacters of the group $C_{P}(\sig)$ where $P = 1+\CJ$ and $\CJ$ is a $\sig$-invariant nilpotent subalgebra of $\CA$. We start by summarising the construction of the supercharacters of the algebra group $P$; our main reference is \cite{DiaIsa}. Let $\jc$ be the dual group of $\CJ^{+}$, and for every $\lam \in \jc$ and every $x \in P$ define the linear characters $\lam x, x\lam \in \jc$ by the formulas $(\lam x)(a) = \lam(ax\inv)$ and $(x\lam)(a) = \lam(x\inv a)$ for all $a \in \CJ$. These actions are compatible in the sense that $(x\lam)y = x(\lam y)$ for all $x,y \in P$ and all $\lam \in \jc$, and thus $\jc$ decomposes as a disjoint union of \textit{two-sided orbits} $P\lam P$ for $\lam \in \jc$. Furthermore, every two-sided orbit on $\jc$ is a disjoint union of \textit{conjugation orbits} where the \textit{conjugation action} $P\x \jc \to \jc$ is defined by the mapping $(x,\lam) \mapsto x\lam x\inv$. We also observe that $P\lam P$ is an orbit for the natural action of $P\x P$ on $\jc$ given by $(x,y) \cdot \lam = x \lam y\inv$ for all $\lam \in \jc$ and all $x,y \in P$.

The supercharacters of $P$ are in one-to-one correspondence with the two-sided orbits on $\jc$. For every $\lam \in \jc$, the {\it supercharacter} $\hchi_{\lam}$ which corresponds to $P\lam P$ is given by the formula
\begin{equation} \label{eq:superch1}
\hchi_{\lam}(x) = \frac{|P\lam|}{|P\lam P|} \sum_{\mu \in P\lam P} \mu(x-1)
\end{equation}
for all $x \in P$; we set $\sch(P) = \set{\hchi_{\lam}}{\lam \in \jc}$. [As for superclasses, we shall use the hat notation $\hchi$ for characters of $P$, and reserve the non-hat notation for the characters of $C_{P}(\sig)$; in particular, $\hchi_{\lam}$ will always refer to the supercharacter of $P$ associated with the linear character $\lam \in \jc$ of $\CJ$.] It is clear that every supercharacter $\hchi \in \sch(P)$ has a constant value on each superclass of $P$, and that for every $\lam, \mu \in \jc$ we have $\frob{\hchi_{\lam}}{\hchi_{\mu}} = 0$ unless $P\lam P = P\mu P$, in which case we clearly have $\hchi_{\lam} = \hchi_{\mu}$. [If $G$ is any finite group, we define the \textit{Frobenius scalar product} $$\frob{\alp}{\beta} = \frac{1}{|G|} \sum_{x \in G} \alp(x)\overline{\beta(x)}$$ for all complex-valued functions $\alp$ and $\beta$ defined on $G$.] In fact, it is straightforward to check that the regular character $\vrho_{P}$ of $P$ decomposes as the orthogonal sum $\vrho_{P} = \sum_{\hchi \in \sch(P)} n_{\hchi} \hchi$ where $n_{\hchi} = \hchi(1)/\frob{\hchi}{\hchi}$ for all $\hchi \in \sch(P)$. In fact, for every $\lam \in \jc$, we have $$\hchi_{\lam}(1) = |P\lam| \and \frob{\hchi_{\lam}}{\hchi_{\lam}} = |P\lam \cap \lam P|$$ (see \cite[Lemma~5.9]{DiaIsa}), and thus if we define $$n_{\lam} = n_{\hchi_{\lam}} = \frac{|P\lam|}{|P\lam \cap \lam P|} = \frac{|P\lam P|}{|P\lam|},$$ then since $\rho_{P} = \sum_{\hphi \in \irr(P)} \hphi(1) \hphi$ we conclude that
\begin{equation} \label{eq:schar}
n_{\lam}\hchi_{\lam} = \sum_{\hphi \in \irr_{\lam}(P)} \hphi(1) \hphi
\end{equation}
where $\irr_{\lam}(P)$ denotes the set consisting of all irreducible constituents of $\hchi_{\lam}$. In particular, it follows that every irreducible character of $P$ is a constituent of a unique supercharacter. Therefore, in order to have a supercharacter theory, it remains to show that a supercharacter is indeed a character of $P$, and this is proved  in \cite[Theorems~5.4~and~5.6]{DiaIsa} (see also \refs{linear}) .

In order to define the supercharacters of the $\sig$-fixed subgroup $C_{P}(\sig)$, we consider $\sig$-invariant supercharacters of $P$; we observe that, if $\cf(P)$ denotes the complex vector space consisting of all class functions of $P$, then $\sig$ acts naturally on $\cf(P)$ by the rule $\psi^{\sig}(x) = \psi(x^{\sig})$ for all $\psi \in \cf(P)$ and all $x \in P$. For our purposes, it is convenient to define the supercharacters of $P$ by means of the inverse Cayley transform $\map{\VPsi}{P}{\CJ}$ as follows. For every $\lam \in \jc$, we define the function $\map{\hxi_{\lam}}{P}{\C}$ by the formula
\begin{equation} \label{eq:superch2}
\hxi_{\lam}(x) = \frac{|P\lam|}{|P\lam P|} \sum_{\mu \in P\lam P} \mu(\VPsi(x))
\end{equation}
for all $x \in P$. We have the following result (which allows us to use the word ``supercharacter'' when we refer to any of these functions).  

\begin{proposition} \label{prop:sch0}
Let $\CJ$ be a $\sig$-invariant nilpotent subalgebra of $\CA$, and let $P = 1+\CJ$. Then, for every $\lam \in \jc$ the supercharacter $\hchi_{\lam} \in \sch(P)$ equals the function $\hxi_{\lam^{2}}$. In particular,  we have $\sch(P) = \set{\hxi_{\lam}}{\lam \in \jc}$.
\end{proposition}

\begin{proof}
By \refeq{superch1}, we see that $\hxi_{\lam}(\VPhi(a)) = \hchi_{\lam}(1+a)$ for all $a \in \CJ$. In fact, since supercharacters are constant on superclasses, \refeq{cayley1} implies that
\begin{align*}
\hchi_{\lam}(\VPhi(a)) &= \hchi_{\lam}(1+2a(1-a)\inv) = \hchi_{\lam}(1+2a) \\ &= \frac{|P\lam|}{|P\lam P|} \sum_{\mu \in P\lam P} \mu(2a) = \frac{|P\lam|}{|P\lam P|} \sum_{\mu \in P\lam P} \mu^{2}(a)
\end{align*}
for all $a \in \CJ$. Since $(x\lam y)^{2} = x\lam^{2}y$ for all $x,y \in P$ (as it is easily seen), it follows that $$\hchi_{\lam}(\VPhi(a)) = \frac{|P\lam|}{|P\lam P|} \sum_{\mu \in P\lam^{2} P} \mu(a) = \hxi_{\lam^{2}}(\VPhi(a))$$ for all $a \in \CJ$ as required.
\end{proof}

We next show that the $\sig$-action on $\cf(P)$ restricts to a $\sig$-action on $\sch(P)$. We first observe that $$P\lam^{\sig} P = (P\lam P)^{\sig} = \set{\mu^{\sig}}{\mu \in P\lam P};$$ in fact, $(x\lam y)^{\sig} = y^{-\sig}\lam^{\sig}x^{-\sig}$ for all $x,y \in P$.

\begin{lemma} \label{lem:sch1}
Let $\CJ$ be a $\sig$-invariant nilpotent subalgebra of $\CA$, let $P = 1+J$, and let $\lam \in \jc$. Then, $(\hxi_{\lam})^{\sig} = \hxi_{\lam^{\sig}}$, and thus $(\hxi_{\lam})^{\sig}$ is a supercharacter of $P$.
\end{lemma}

\begin{proof}
If $x \in P$, then $\VPsi(x)^{\sig} = \VPsi(x^{\sig})$, and so we deduce that
\begin{align*}
(\hxi_{\lam})^{\sig}(x) &= \hxi_{\lam}(x^{\sig}) = \frac{|P\lam|}{|P\lam P|} \sum_{\mu \in P\lam P} \mu(\VPsi(x^{\sig})) = \frac{|P\lam|}{|P\lam P|} \sum_{\mu \in P\lam P} \mu(\VPsi(x)^{\sig}) \\ &= \frac{|P\lam|}{|P\lam P|} \sum_{\mu \in P\lam P} \mu^{\sig}(\VPsi(x)) = \frac{|P\lam|}{|P\lam P|} \sum_{\mu \in P\lam^{\sig} P} \mu(\VPsi(x)) = \hxi_{\lam^{\sig}}(x)
\end{align*}
as required.
\end{proof} 

We denote by $\sch_{\sig}(P)$ the subset of $\sch(P)$ consisting of all $\sig$-invariant supercharacters. The following result describes this subset; we recall that $\jcc = \set{\lam \in \jc}{\lam^{\sig} = \lam}$.

\begin{proposition} \label{prop:sch2}
Let $\CJ$ be a $\sig$-invariant nilpotent subalgebra of $\CA$, and let $P = 1+J$. Then, $\sch_{\sig}(P) = \set{\hxi_{\lam}}{\lam \in \jcc}$.
\end{proposition}

\begin{proof}
By the previous lemma, we see that $(\hxi_{\lam})^{\sig} = \hxi_{\lam^{\sig}} = \hxi_{\lam}$, and thus $\hxi_{\lam} \in \sch_{\sig}(P)$ for all $\lam \in \jcc$. Conversely, let $\mu \in \jc$ be such that $\hxi_{\mu} \in \sch_{\sig}(P)$. Since $\hxi_{\mu} = (\hxi_{\mu})^{\sig} = \hxi_{\mu^{\sig}}$, we conclude that $\mu^{\sig} \in P\mu P$, and this clearly implies that the two-sided orbit $P\mu P$ is $\sig$-invariant. Now, we consider $\sig$ as the automorphism of $P\x P$ given by $(x,y)^{\sig} = (y^{\sig},x^{\sig})$ for all $x,y \in P$, and observe that $$((x,y)\cdot \nu)^{\sig} = (x\inv\nu y)^{\sig} = y^{-\sig}\nu^{\sig}x^{\sig} = (y^{\sig},x^{\sig}) \cdot \nu^{\sig} = (x,y)^{\sig} \cdot \nu^{\sig}$$ for all $x,y \in P$ and all $\nu \in \jc$. Thus, since $P\x P$ acts transitively on $P\mu P$ (and since $2 \nmid |P|$), Glauberman's Lemma (see \cite[Lemma~13.8]{Isa1}) implies that there exists $\lam \in P\mu P$ such that $\lam^{\sig} = \lam$. Since $\hxi_{\lam} = \hxi_{\mu}$, the result follows.
\end{proof}

As in the case of superclasses, it is natural to expect that supercharacters of $C_{P}(\sig)$ would be in one-to-one correspondence with $\sig$-invariant two-sided orbits of $P$ on $\CJ$, and in fact we shall prove that a given supercharacter is determined by the subset consisting of the $\sig$-fixed elements in the corresponding $\sig$-invariant two-sided orbit. For any $\lam \in \jcc$, we define $\Ome_{P}(\lam)$ to be the subset of $P\lam P$ consisting of all $\sig$-fixed elements.

\begin{proposition} \label{prop:fixed}
Let $\CJ$ be a $\sig$-invariant nilpotent subalgebra of $\CA$, and let $P = 1+\CJ$. Then, $$\Ome_{P}(\lam) = \set{x\inv \lam x^{\sig}}{x \in P}$$ for all $\lam \in \jcc$.
\end{proposition}

\begin{proof}
As before, we consider $\sig$ as an automorphism of $P\x P$. By \cite[Corollary~13.9]{Isa1}, the set of $\sig$-fixed elements of $P\lam P$ is an orbit under the action of $C_{P\x P}(\sig)$, and the result follows because $C_{P\x P}(\sig) = \set{(x,x^{\sig})}{x \in P}$.
\end{proof}

Next, we consider the {\it Glauberman correspondence} between $\sig$-invariant irreducible characters of $P$ and irreducible characters of $C_{P}(\sig)$; our main reference is \cite[Chapter~13]{Isa1}. Since $p$ is odd, this correspondence asserts that there exists a uniquely defined bijective map $$\map{\pi_{P}}{\irr_{\sig}(P)}{\irr(C_{P}(\sig))}$$ such that, for any $\hchi \in \irr_{\sig}(P)$, the image $\chi = \pi_{P}(\hchi)$ is the unique irreducible constituent of the restriction $\hchi_{C_{P}(\sig)}$ with odd multiplicity (see \cite[Theorem~13.1]{Isa1});  here, and henceforth, we denote by $\irr_{\sig}(P)$ the subset of $\irr(P)$ consisting of all $\sig$-invariant irreducible characters of $P$.

\begin{lemma} \label{lem:glaub}
Let $\CJ$ be a $\sig$-invariant nilpotent subalgebra of $\CA$, and let $P = 1+\CJ$. Let $\chi$ be any irreducible character of $C_{P}(\sig)$, let $\hchi \in \irr_{\sig}(P)$ be such that $\pi_{P}(\hchi) = \chi$, and let $\hxi \in \sch(P)$ be the unique supercharacter such that $\frob{\hchi}{\hxi} \neq 0$. Then, $\hxi^{\sig} = \hxi$, and in particular there exists $\lam \in \jcc$ such that $\hxi = \hxi_{\lam}$.
\end{lemma}

\begin{proof}
This is an immediate consequence of the orthogonality of supercharacters because $\hchi = \hchi^{\sig}$ is an irreducible constituent of the supercharacter $\hxi^{\sig}$ of $P$.
\end{proof}

For any $\lam \in \jcc$, we write $X(\lam)$ to denote the set consisting of all irreducible characters $\chi \in \irr(C_{P}(\sig))$ such that Glauberman correspondent $\hchi \in \irr_{\sig}(P)$ of $\chi$ is a constituent of the supercharacter $\hxi_{\lam} \in \sch_{\sig}(P)$, and define
\begin{equation} \label{eq:sig2}
\sig_{\lam} = \sum_{\chi \in X(\lam)} \chi(1) \chi;
\end{equation}
notice that this is precisely the character $\sig_{X(\lam)}$ of $C_{P}(\sig)$ defined in \refeq{sig1}.

\begin{theorem} \label{thm:super1}
Let $\CJ$ be a $\sig$-invariant nilpotent subalgebra of $\CA$, and let $P = 1+\CJ$. Then, $\bset{X(\lam)}{\lam \in \jcc}$ is a partition of $\irr(C_{P}(\sig))$; in particular, every irreducible character $\chi \in \irr(C_{P}(\sig))$ is a constituent of $\sig_{\lam}$ for some $\lam \in \jcc$. Furthermore, for every $\lam, \mu \in \jcc$ we have $\sig_{\lam} = \sig_{\mu}$ if and only if $\mu \in \Ome_{P}(\lam)$.
\end{theorem}

\begin{proof}
By the previous lemma, it is clear that $\irr(C_{P}(\sig))$ is the union $$\irr(C_{P}(\sig)) = \bigcup_{\lam \in \jcc} X(\lam).$$ To show that this union is disjoint, let $\chi \in X(\lam) \cap X(\mu)$ for $\lam, \mu \in \jcc$, and let $\hchi \in \irr_{\sig}(P)$ be such that $\chi = \pi_{P}(\hchi)$. Then, $\hchi$ is a common irreducible constituent of the supercharacters $\hxi_{\lam}, \hxi_{\mu} \in \sch(P)$, and thus $\hxi_{\lam} = \hxi_{\mu}$ (by the orthogonality of supercharacters). It follows that $\mu \in P\lam P$, and the result is now a consequence of \refeq{jcc} and \refp{fixed}.
\end{proof}

As a consequence of this theorem, we see that $$X(\lam) = \set{\pi_{P}(\hchi)}{\hchi \in \irr_{\sig}(P),\ \frob{\hchi}{\hxi_{\lam}} \neq 0}$$ for all $\lam \in \jcc$. Furthermore, the theorem suggests that, if $\CX = \bset{X(\lam)}{\lam \in \jcc}$ and $\CK = \scl(C_{P}(\sig))$ (as in \refeq{scl}), then the pair $(\CX,\CK)$ forms a supercharacter theory for the $\sig$-fixed subgroup $C_{P}(\sig)$; alternatively, we may define
\begin{equation} \label{eq:sch}
\sch(C_{P}(\sig)) = \set{\sig_{\lam}}{\lam \in \jcc}
\end{equation}
as the set of supercharacters of $C_{P}(\sig)$. Further evidence is given by the following result.

\begin{theorem} \label{thm:scf1}
Let $\CJ$ be a $\sig$-invariant nilpotent subalgebra of $\CA$, and let $P = 1+\CJ$. Then, the sets
\begin{itemize}
\item $\CX = \bset{X(\lam)}{\lam \in \jcc}$,
\item $\sch(C_{P}(\sig)) = \set{\sig_{\lam}}{\lam \in \jcc}$, and
\item $\scl(C_{P}(\sig)) = \set{\HCK \cap C_{P}(\sig)}{\HCK \in \scl_{\sig}(P)}$
\end{itemize}
have the same cardinality.
\end{theorem}

\begin{proof}
By the previous theorem, it is obvious that $|\CX| = |\sch(C_{P}(\sig))|$. To prove of the other equality, we consider the action of $P$ on $C_{\CJ}(\sig)$ given by $x\cdot a = xax^{-\sig}$ for all $x \in P$ and all $a \in C_{\CJ}(\sig)$, and denote by $\Omega$ the set consisting of all orbits of $P$ on $C_{\CJ}(\sig)$; notice that $|\Omega| = |\scl(C_{P}(\sig))|$ (by \reft{scl2}). On the other hand, we also consider the contragradient action of $P$ on the dual group $\jcc$ given by $x\cdot \lam = x \lam x^{-\sig}$ or all $x \in P$ and all $\lam \in \jcc$, and denote by $\Omega^{\circ}$ the set consisting of all orbits of $P$ on $\jcc$. By \reft{super1} we have $|\Omega^{\circ}| = |\CX|$, and thus we must prove that $|\Omega| = |\Omega^{\circ}|$. To see this, let $\tau$ be the permutation character of $P$ on $C_{\CJ}(\sig)$; hence, $\tau(x) = |\set{a \in C_{\CJ}(\sig)}{x\cdot a = a}|$ for all $x \in P$. Since $(x\cdot\lam)(x\cdot a) = \lam(a)$ for all $x \in P$, all $\lam \in \jcc$ and all $a \in C_{\CJ}(\sig)$, it follows from Brauer's Theorem (\cite[Theorem~6.32]{Isa1}) that $\tau(x) = |\set{\lam \in \jcc}{x\cdot\lam = \lam}|$ for all $x \in P$, and thus $\tau$ is also the permutation character of $P$ on $\jcc$. By \cite[Corollary~5.15]{Isa1}, we conclude that $|\Omega| = \frob{\tau}{1_{P}} = |\Omega^{\circ}|$ as required.
\end{proof}

Thus, in order to establish that we have a genuine supercharacter theory for $C_{P}(\sig)$ only one thing remains: we must show that for every $\lam \in \jcc$, the (super)character $\sig_{\lam}$ is a superclass function. This will be a consequence of the following main result which gives a convenient way to compute the values of a supercharacter.

\begin{theorem} \label{thm:super2}
Let $\CJ$ be a $\sig$-invariant nilpotent subalgebra of $\CA$, let $P = 1+\CJ$, and let $\lam \in \jcc$. Then,
\begin{equation} \label{eq:sig3}
\sig_{\lam}(x) = \sum_{\mu \in \Ome_{P}(\lam)} \mu(\VPsi(x))
\end{equation}
for all $x \in C_{P}(\sig)$. In particular, $\sig_{\lam}$  has a constant value on each superclass of $C_{P}(\sig)$.
\end{theorem}

The proof of this theorem will be the main goal of the next two sections. Once the theorem is proved, then we can define for every $\lam \in \jcc$ the \textit{supercharacter} of $C_{P}(\sig)$ associated with $\lam$ to be the function $\map{\vsig_{\lam}}{C_{P}(\sig)}{\C}$ by the rule
\begin{equation} \label{eq:zet2}
\vsig_{\lam}(x) = \sum_{\mu \in \Ome_{P}(\lam)} \mu(\VPsi(x))
\end{equation}
for all $x \in C_{P}(\sig)$; notice that $\vsig_{\lam}$ depends only on the orbit $\Ome_{P}(\lam) = \set{x\inv \lam x^{\sig}}{x \in P}$ where we consider the action of $P$ on the left of $\jcc$ given by $x\cdot \lam = x\inv \lam x^{\sig}$ for all $x \in P$ and all $\lam \in \jcc$. It is clear that $\vsig_{\lam}(yxy\inv) = \vsig_{\lam}(x)$ for all $x,y \in C_{P}(\sig)$, and hence $\vsig_{\lam}$ is a class function of $C_{P}(\sig)$. Since $\irr(C_{P}(\sig))$ is a $\C$-basis of $\cf(C_{P}(\sig))$, it follows that $\vsig_{\lam}$ it is a $\C$-linear combination of the irreducible characters of $C_{P}(\sig)$. Our aim is to prove that $\vsig_{\lam}$ is a character of $C_{P}(\sig)$, and this occurs if and only if it is a linear combination of irreducible characters with positive integer coeficients. In fact, \reft{super2} claims that $\vsig_{\lam} = \sig_{\lam} = \sum_{\chi \in X(\lam)} \chi(1) \chi$, and thus we must prove that an irreducible character $\chi \in \irr(C_{P}(\sig))$ appears in the class function $\vsig_{\lam} \in \cf(C_{P}(\sig))$ (with non-zero multiplicity) if and only if its Glauberman correspondent $\hchi \in \irr_{\sig}(P)$ appears in the supercharacter $\hxi_{\lam} \in \sch(P)$ (with non-zero multiplicity); moreover, if this is the case, then we must also show that the multiplicity $\frob{\chi}{\sig_{\lam}}$ equals the degree $\chi(1)$ of $\chi$. To achieve this, we recall that by \refeq{schar} (see also \cite[Theorem~5.5(ii)]{DiaIsa} and \refp{sch0}) we have $$n_{\lam}\hxi_{\lam} =\sum_{\hchi \in \irr_{\lam}(P)} \hchi(1) \hchi$$ where $n_{\lam} = |P\lam P|/|P\lam|$; furthermore, it follows from \cite[Theorem~5.6]{DiaIsa} (and from \refp{sch0}) that $n_{\lam}\hxi_{\lam} = \hsig_{\lam}$ where $\map{\hsig_{\lam}}{P}{\C}$ is the function defined by the rule
\begin{equation} \label{eq:sig4}
\hsig_{\lam}(x) = \sum_{\mu \in P\lam P} \mu(\VPsi(x))
\end{equation}
for all $x \in P$. On the other hand, if $\hchi \in \irr_{\sig}(\lam)$ is an arbitrary $\sig$-invariant irreducible constituent of $\hxi_{\lam}$, then \cite[Theorem~2.1]{And4} asserts that there exist a $\sig$-invariant algebra subgroup $Q$ of $P$ and a $\sig$-invariant linear character $\htet \in \irr_{\sig}(Q)$ such that $\hchi = \htet^{P}$ and $\chi = \tet^{C_{P}(\sig)}$ where $\chi = \pi_{P}(\hchi) \in \irr(C_{P}(\sig))$ and $\tet = \pi_{Q}(\htet) \in \irr(C_{Q}(\sig))$; given any $\sig$-invariant subgroup $Q$ of $P$, we write $\pi_{Q}$ to denote the Glauberman map $\map{\pi_{Q}}{\irr_{\sig}(Q)}{\irr(C_{Q}(\sig))}$. By the above (and by Frobenius reciprocity), we have $\hchi(1) = \frob{\chi}{n_{\lam} \hxi_{\lam}} = \frob{\htet^{P}}{\hsig_{\lam}} = \frob{\htet}{(\hsig_{\lam})_{Q}}$. By \cite[Theorem~6.4]{DiaIsa}, the restriction $(\hxi_{\lam})_{Q}$ decomposes as a sum of supercharacters of $Q$, and hence $(\hsig_{\lam})_{Q} = n_{\lam}(\hxi_{\lam})_{Q}$ also decomposes as a sum of supercharacters of $Q$. It follows that there exists a unique supercharacter $\hxi_{0} \in \sch(Q)$ such that $\hxi_{0}$ is a constituent of $(\hsig_{\lam})_{Q}$ and $\htet$ is a constituent of $\hxi_{0}$. In light of this reduction process, we will first prove \reft{super2} in the more favourable situation where the supercharacter $\hxi_{\lam}$ has a linear constituent.
 

\section{Supercharacters with a linear constituent} \label{sec:linear}

As before, let $\CJ$ be a $\sig$-invariant nilpotent subalgebra of $\CA$, and let $P = 1+\CJ$. Our aim is to prove \reft{super2} in the particular situation where $\lam \in \jcc$ is such that $\hxi_{\lam} \in \sch(P)$ has a linear constituent. We start by recalling some general facts about the supercharacter $\hxi_{\lam}$. We define $$\CL(\lam) = \set{a\in \CJ}{a\CJ \sset \ker(\lam)}\quad \text{and}\quad L(\lam) = 1+\CL(\lam).$$ Then, $\CL(\lam)$ is a right ideal (hence, a subalgebra) of $\CJ$, and thus $L(\lam)$ is an algebra subgroup of $P$; notice that $L(\lam) = \set{x \in P}{x\lam = \lam}$ is the centralizer of $\lam$ with respect to the left action of $P$ on $\CJ$. The mapping $x \mapsto \lam(x-1)$ clearly defines a linear character $\map{\htau_{\lam}}{L(\lam)}{\cx}$, and it is proved in \cite[Theorems~5.4~and~5.6]{DiaIsa} that $\hchi_{\lam} = (\htau_{\lam})^{P}$; recall that we are writing $\hchi_{\lam}$ for the (super)character of $P$ defined by \refeq{superch1}. [In particular, we conclude that $\hchi_{\lam}$ is indeed a character of $P$.] Next, we prove that the supercharacter $\hxi_{\lam}$ is also induced from a linear character of the subgroup $L(\lam)$. In fact, since the Cayley transform $\map{\VPhi}{\CJ}{P}$ clearly maps $\CL(\lam)$ to $L(\lam)$ bijectively, we may define the function $\map{\htet_{\lam}}{L(\lam)}{\cx}$ by the rule
\begin{equation} \label{eq:lin1}
\htet_{\lam}(x) = \lam(\VPsi(x))
\end{equation}
for all $x \in L(\lam)$. Then, we obtain the following result (where we are not assuming that the supercharacter $\hxi_{\lam} \in \sch(P)$ has a linear constituent).

\begin{proposition} \label{prop:lin1}
Let $\CJ$ be a $\sig$-invariant nilpotent subalgebra of $\CA$, and let $P = 1+\CJ$. Then, for every $\lam \in \jc$ the function $\htet_{\lam}$ is a linear character of $L(\lam)$, and we have $\hxi_{\lam} = (\htet_{\lam})^{P}$.
\end{proposition}

\begin{proof}
By the definition of $\CL(\lam)$, it is clear that $\lam(ax) = \lam(a)$ for all $a \in \CL(\lam)$ and all $x \in P$. On the other hand, let $\mu \in \jc$ be such that $\lam = \mu^{2}$. Then, $\lam(a) = \mu(2a)$ for all $a \in \CJ$, and thus $\mu(ax) = \mu(a)$ for all $a \in \CL(\lam)$ and all $x \in P$; in fact, we have $\CL(\mu) = \CL(\lam)$. In particular, we deduce that $$\htet_{\lam}(\VPhi(a)) = \lam(a) = \mu(2a) = \mu(2a(1-a)\inv) = \mu(\VPhi(a)-1)$$ for all $a \in \CL(\lam)$, and thus $\htet_{\lam}(x) = \htau_{\mu}(x)$ for all $x \in L(\lam)$. It follows that $\htet_{\lam}$ is a linear character of $L(\lam)$, and that $(\htet_{\lam})^{P} = (\htau_{\mu})^{P} = \hchi_{\mu} = \hxi_{\lam}$ (by \refp{sch0}).
\end{proof}

Under our assumption that $\hxi_{\lam} \in \sch(P)$ has a linear constituent, \cite[Corollary~5.12]{DiaIsa} assures that $\CL(\lam)$ is a two-sided ideal of $\CJ$, and hence $L(\lam)$ is a normal subgroup of $P$; furthermore, we have $P\lam = \lam P = P\lam P$, and thus $$\hxi_{\lam}(x) = \hsig_{\lam}(x) = \sum_{\mu \in P\lam P} \mu(\VPsi(x))$$ for all $x \in \CJ$. On the other hand, we observe that the subgroup $L(\lam)$ is $\sig$-invariant: in fact, since $\lam$ is $\sig$-invariant, we have $(x\lam)^{\sig} = \lam x^{\sig}$ for all $x \in P$, and thus $L(\lam)^{\sig} = L(\lam)$ (again by \cite[Corollary~5.12]{DiaIsa}). We now consider the $\sig$-fixed subgroup $C_{L(\lam)}(\sig)$, and note that $C_{L(\lam)}(\sig) = \VPhi(C_{\CL(\lam)}(\sig))$ where $C_{\CL(\lam)}(\sig) = \set{a \in C_{\CJ}(\sig)}{a^{\sig} = a}$. We define the linear character $\map{\tet_{\lam}}{C_{L(\lam)}(\sig)}{\cx}$ to be the restriction of $\htet_{\lam}$ to $C_{L(\lam)}(\sig)$; hence,
\begin{equation} \label{eq:lin2}
\tet_{\lam}(x) = \lam(\VPsi(x))
\end{equation}
for all $x \in C_{L(\lam)}(\sig)$. Furthermore, we define $\xi_{\lam}$ to be the induced character
\begin{equation} \label{eq:suplin}
\xi_{\lam} = (\tet_{\lam})^{C_P(\sigma)}.
\end{equation}
The following result is a simple consequence of \cite[Theorem~13.29]{Isa1}; we recall that $L(\lam)$ is a normal subgroup of $P$.

\begin{lemma} \label{lem:sch2}
Let $\CJ$ be a $\sig$-invariant nilpotent subalgebra of $\CA$, let $P = 1+\CJ$, and let $\lam \in \jcc$ be such that the supercharacter $\hxi_{\lam} \in \sch(P)$ has a linear constituent. Let $\chi \in \irr(C_{P}(\sig))$, and let $\hchi \in \irr_{\sig}(P)$ be such that $\pi_{P}(\hchi) = \chi$. Then, $\frob{\chi}{\xi_{\lam}} \neq 0$ if and only if $\frob{\hchi}{\hxi_{\lam}} \neq 0$.
\end{lemma}

We are now able to prove the following particular case of \reft{super2}.

\begin{theorem} \label{thm:super3}
Let $\CJ$ be a $\sig$-invariant nilpotent subalgebra of $\CA$, let $P = 1+\CJ$, and let $\lam \in \jcc$ be such that the supercharacter $\hxi_{\lam} \in \sch(P)$ has a linear constituent. Then, $$\sig_{\lam}(x) = \sum_{\mu \in \Ome_{P}(\lam)} \mu(\VPsi(x))$$ for all $x \in C_{P}(\sig)$. Moreover, we have $\sig_{\lam} = \xi_{\lam} = (\tet_{\lam})^{C_{P}(\sig)}$.
\end{theorem}

\begin{proof}
Let $x \in C_{P}(\sig)$ be arbitrary. We show that both members of the desired equality are equal to $0$ unless $x \in C_{L(\lam)}(\sig)$ in which case they are both equal to $|C_{P}(\sig) : C_{L(\lam)}(\sig)|\,\tet_{\lam}(x)$.

To start with, we observe that this is precisely the value $\xi_{\lam}(x) = (\tet_{\lam})^{C_{P}(\sig)}(x)$. In fact, since $L(\lam)$ is a normal subgroup of $P$, $C_{L(\lam)}(\sig) = L(\lam) \cap C_{P}(\sig)$ is a normal subgroup of $C_{P}(\sig)$, and thus $(\tet_{\lam})^{C_{P}(\sig)}(x) = 0$ whenever $x \notin C_{L(\lam)}(\sig)$. On the other hand, by \cite[Corollary~4.3]{AndNic} the linear character $\map{\htet_{\lam}}{L(\lam)}{\cx}$ is $P$-invariant, and so its restriction to $C_{L(\lam)}(\sig)$ is $C_{P}(\sig)$-invariant. It follows that $(\tet_{\lam})^{C_{P}(\sig)}(x) = |C_{P}(\sig):C_{L(\lam)}(\sig)|\,\tet_{\lam}(x)$ whenever $x \in C_{L(\lam)}(\sig)$.

Next, we show that
\begin{equation} \label{eq:supch2}
\xi_{\lam}(x) = \sum_{\mu \in \Ome_{P}(\lam)} \mu(\VPsi(x)).
\end{equation}
By \cite[Lemma~4.2]{DiaIsa}, we have $P\lam = \lam + \CL(\lam)\ort$ where $\CL(\lam)\ort = \set{\nu \in \jc}{\CL(\lam) \sset \ker(\nu)}$. Since $\Ome_{P}(\lam) = \set{\mu \in P\lam P}{\mu^{\sig} = \mu}$ (by \refp{fixed}) and since $P\lam P = P\lambda = \lam P$, we conclude that $\Ome_{P}(\lam) = \lam + \set{\nu \in \CL(\lam)\ort}{\nu^{\sig} =\nu}$. If $\nu \in \jc$, then $\nu^{\sig} = \nu$ if and only if $\nu \in \jcc$, and so $\set{\nu \in \CL(\lam)\ort}{\nu^{\sig} =\nu} = \jcc \cap \CL(\lam)\ort$; moreover, $\CL(\lam) \sset \ker(\nu)$ if and only if $C_{\CL(\lam)}(\sig) \sset \ker(\nu)$, and thus $\jcc \cap \CL(\lam)\ort = C_{\CL(\lam)}(\sig)\ort$ where we set $C_{\CL(\lam)}(\sig)\ort = \set{\nu \in \jcc}{C_{\CL(\lam)}(\sig) \sset \ker(\nu)}$. It follows that $\Ome_{P}(\lam) = \lam + C_{\CL(\lam)}(\sig)\ort$, and thus $$\sum_{\mu \in \Ome_{P}(\lam)} \mu(a) = \lam(a) \sum_{\nu \in C_{\CL(\lam)}(\sig)\ort} \nu(a)$$ where $a = \VPsi(x) \in C_{\CJ}(\sig)$. Since the sum $\sum_{\nu \in C_{\CL(\lam)}(\sig)\ort} \nu$ naturally identifies with the regular character of the additive group $C_{\CJ}(\sig)/C_{\CL(\lam)}(\sig)$, we conclude that $$\sum_{\mu \in \Ome_{P}(\lam)} \mu(a) = \begin{cases} 0, & \text{if $a \notin C_{\CL(\lam)}(\sig)$,} \\ |C_{\CJ}(\sig):C_{\CL(\lam)}(\sig)|\, \lam(a), & \text{if $a \in C_{\CL(\lam)}(\sig)$,} \end{cases}$$
and \refeq{supch2} follows because the Cayley transform $\map{\VPhi}{\CJ}{P}$ is bijective and maps $C_{\CJ}(\sig)$ to $C_{P}(\sig)$ and $C_{\CL(\lam)}(\sig)$ to $C_{L(\lam)}(\sig)$.

To conclude the proof, we apply Gallagher's Theorem (see \cite[Corollary~6.17]{Isa1}) to identify the irreducible constituents of $\xi_{\lam} = (\tet_{\lam})^{C_{P}(\lam)}$; we recall that $C_{L(\lam)}(\sig)$ is a normal subgroup of $C_{P}(\sig)$. We first claim that the linear character $\tet_{\lam}$ of $C_{L(\lam)}(\sig)$ extends to $C_{P}(\sig)$. To see this, let $\htau \in \irr(P)$ be a linear constituent of $\hxi_{\lam}$ (which exists by assumption), and let $\tau$ be its restriction to $C_{P}(\sig)$. (Notice that $\htau$ is not necessarily $\sig$-invariant, hence it may not be the Glauberman correspondent of $\tau$.) Since $\htet_{\lam}$ is $P$-invariant, we have $\htau_{L(\lam)} = \htet_{\lam}$, and hence $$\tau_{C_{L(\lam)}(\sig)} = (\htet_{\lam})_{C_{L(\lam)}(\sig)} = \tet_{\lam}.$$ Therefore, $\tau$ is an extension of $\tet_{\lam}$ to $C_{P}(\sig)$, and so  Gallagher's Theorem implies that $$\xi_{\lam} = (\tet_{\lam})^{C_{P}(\sig)} = \sum_{\begin{smallmatrix} \ome \in \irr(C_{P}(\sig)) \\ C_{L(\lam)}(\sig) \sset \ker(\ome) \end{smallmatrix}} \ome(1) (\tau\ome).$$ Finally, it easily seen from \refp{sch2} that $$X(\lam) = \set{\tau\ome}{\psi \in \irr(C_{P}(\sig)),\ C_{L(\lam)}(\sig) \sset \ker(\psi)},$$ and thus $$\xi_{\lam} = \sum_{\chi \in X(\lam)} \chi(1) \chi = \sig_{\lam}.$$

The proof is complete.
\end{proof} 


\section{Proof of \reft{super2}} \label{sec:proof}

Let $\CJ$ a $\sig$-invariant nilpotent subalgebra of $\CA$, and let $P = 1+\CJ$. Otherwise stated, we fix a linear character $\lam \in \jcc$ throughout the section. Our primary goal is to show that \refeq{sig3} holds, and we shall use the reduction process described before. We let $\hchi \in \irr_{\sig}(\lam)$ be an arbitrary $\sig$-invariant irreducible constituent of the supercharacter $\hxi_{\lam} \in \sch(P)$, and choose a $\sig$-invariant algebra subgroup $Q$ of $P$ and a $\sig$-invariant linear character $\htet$ of $Q$ such that $$\hchi = \htet^{P}\quad \text{and}\quad \chi = \tet^{C_{P}(\sig)}$$ where $\chi = \pi_{P}(\hchi) \in \irr(C_{P}(\sig))$ and $\tet = \pi_{Q}(\htet) \in \irr(C_{Q}(\sig))$ (the existence of $Q$ and $\htet$ is guaranted by \cite[Theorem~2.1]{And4}). Then, $\hchi(1) = \frob{\chi}{\hsig_{\lam}} = \frob{\htet}{(\hsig_{\lam})_{Q}}$, and thus there exists a unique supercharacter $\hxi_{0} \in \sch(Q)$ such that $\hxi_{0}$ is a constituent of $(\hsig_{\lam})_{Q}$ and $\htet$ is a constituent of $\hxi_{0}$; recall that the restriction $(\hsig_{\lam})_{Q} = n_{\lam}(\hxi_{\lam})_{Q}$ decomposes as a sum of supercharacters of $Q$ (by \cite[Theorem~6.4]{DiaIsa}). We now prove the following result (which holds for every algebra group).

\begin{proposition} \label{prop:lin}
Let $P = 1+\CJ$ be an algebra group over $\k$, and let $\hchi \in \irr(P)$ be an irreducible constituent of a supercharacter $\hxi \in \sch(P)$. Let $\CI$ be a subalgebra of $\CJ$, let $Q = 1+\CI$, and suppose that $\hchi = \htet^{P}$ for some a linear character $\htet$ of $Q$. Let $\hxi_{0} \in \sch(Q)$ be the unique supercharacter of $Q$ such that $\htet$ is a constituent of $\hxi_{0}$. Then:
\begin{enumerate}
\item $\hxi_{0}$ is a constituent of the restriction $\hxi_{Q}$ with multiplicity $\hchi(1)$.
\item There exists $\lam \in \jc$ such that $\hxi = \hxi_{\lam}$ and $\hxi_{0} = \hxi_{\lam_{0}}$ where $\lam_{0} = \lam_{\CI}$ is the restriction of $\lam$ to $\CI$.
\item If $\mu_{0} \in Q\lam_{0} Q$ and $\mu \in \jc$ is such that $\mu_{\CI} = \mu_{0}$, then $\mu + \CI^{\perp} \sset P\lam P$; in particular, the set $\set{\mu \in P\lam P}{\mu_{\CI} = \mu_{0}}$ has cardinality $|\CJ:\CI| = |P:Q|$.
\end{enumerate}
\end{proposition}

\begin{proof}
Since $\htet$ is linear, \cite[Corollary~5.12]{DiaIsa} asserts that $Q\lam_{0} = \lam_{0}Q = Q\lam_{0}Q$, and thus $$\hxi_{0} = \sum_{\hchi_{0} \in \irr_{\lam_{0}}(Q)} \hchi_{0}(1) \hchi_{0}$$ (by \cite[Theorem~5.5(ii)]{DiaIsa}). Since $\hchi(1) = \frob{\htet}{(\hsig_{\lam})_{Q}}$ and $\frob{\htet}{\hxi_{0}} = \htet(1) = 1$, we conclude that $$(\hsig_{\lam})_{Q} = \hchi(1) \hxi_{0} + \hzet$$ where $\hzet$ is a sum of supercharacters of $Q$ all distinct from $\hxi_{0}$; in particular, we have $\frob{\htet}{\hzet} = 0$. By the definition of $\hsig_{\lam}$ and of $\hxi_{0}$ (see \refeq{superch2}), we deduce that $$\sum_{\mu \in P\lam P} \mu_{\CI} = \hchi(1) \sum_{\mu_{0} \in Q\lam_{0} Q} \mu_{0} + \mu'$$ where $\mu' $ is a character (not necessarily linear) of the additive group $\CI^{+}$ satisfying $\frob{\mu'}{\mu_{0}} = 0$ for all $\mu_{0} \in Q\lam_{0}Q$. It follows that every linear character $\mu_{0} \in Q\lam_{0}Q$ occurs with multiplicity $\hchi(1)$ in the sum of the left hand side, and hence the set $\set{\mu \in P\lam P}{\mu_{\CI} = \mu_{0}}$ has cardinality $\hchi(1)$. Since $$\hchi(1) = \htet^{P}(1) = |P:Q| = |\CJ:\CI| = |\CI^{\perp}|,$$ we conclude that $\mu + \CI^{\perp} \sset P\lam P$ for all $\mu \in P\lam P$ such that $\mu_{\CI} \in Q\lam_{0} Q$, and this completes the proof.
\end{proof}

We are now able to proceed with the proof of \reft{super2}.

\begin{proof}[Proof of \reft{super2}]
We assume that $\lam \in \jcc$, and let the notation be as above; without loss of generality, we may assume that $\hxi_{0} = \hxi_{\lam_{0}}$ is the supercharacter of $Q$ corresponding to the restriction $\lam_{0} = \lam_{\CI}$ of $\lam$ to $\CI = Q-1$. Let $\Ome_{Q}(\lam_{0}) = \set{x\inv \lam_{0} x^{\sig}}{x\in Q}$ be the subset of $Q\lam_{0}Q \sset \ic$ consisting of $\sig$-fixed elements, and consider the function $\map{\vsig_{0}}{C_{Q}(\sig)}{\C}$ given by $$\vsig_{0}(x) = \sum_{\mu_{0} \in \Ome_{Q}(\lam_{0})} \mu_{0}(\VPsi(x))$$ for all $x \in C_{Q}(\sig)$. Then, since $\htet \in \irr_{\sig}(Q)$ is a $\sig$-invariant linear constituent of the supercharacter $\hxi_{0} \in \sch(Q)$, \reft{super3} implies that $$\vsig_{0} = \sum_{\chi_{0} \in X(\lam_{0})} \chi_{0}(1) \chi_{0}$$ where $X(\lam_{0}) = \set{\pi_{Q}(\hchi_{0})}{\hchi_{0} \in \irr_{\sig}(Q),\ \frob{\hchi_{0}}{\hxi_{0}} \neq 0}$; in particular, $\tet = \pi_{Q}(\htet) \in \irr(C_{Q}(\sig))$ is a linear constituent of $\vsig_{0}$ occuring with multiplicity one. Our goal is to show that the irreducible character $\chi = \tet^{C_{P}(\sig)}$ appears as a constituent of $\vsig_{\lam}$ with multiplicity $$\chi(1) = |C_{P}(\sig) : C_{Q}(\sig)| = |C_{\CJ}(\sig) : C_{\CI}(\sig)| = |C_{\CI}(\sig)^{\perp}|$$ where $C_{\CI}(\sig)^{\perp} = \set{\nu \in \jcc}{C_{\CI}(\sig) \sset \ker(\nu)}$.

Firstly, observe that \reft{scl2} and \refp{fixed} clearly imply that for all $\nu_{0} \in \icc$ the function $\map{\vsig_{\nu_{0}}}{C_{Q}(\sig)}{\C}$ (defined as in \refeq{zet2}) is constant on each superclass of $C_{Q}(\sig)$; moreover, the proof of \refp{fixed} shows that $\set{\Ome_{Q}(\nu_{0})}{\nu_{0} \in \icc}$ is a partition of $\icc$. It follows that $\set{\vsig_{\nu_{0}}}{\nu_{0} \in \icc}$ is an orthogonal basis of the complex space space $\scf(C_{Q}(\sig))$ consisting of all superclass functions of $C_{Q}(\sig)$. Therefore, since the restriction $(\vsig_{\lam})_{C_{Q}(\sig)}$ of $\vsig_{\lam}$ to $C_{Q}(\sig)$ is clearly a superclass function on $C_{Q}(\sig)$, we conclude that there exist $\seq{\nu}{r} \in \icc$ and $\seq{z}{r} \in \C$ such that $(\vsig_{\lam})_{C_{Q}(\sig)} = z_{1}\vsig_{\nu_{1}} + \cdots + z_{r} \vsig_{\nu_{r}}$ where $\frob{\vsig_{\nu_{i}}}{\vsig_{\nu_{j}}} = 0$ for all $1 \leq i \neq j \leq r$; in other words, we have $$\sum_{\mu \in \Ome_{P}(\lam)} \mu_{C_{\CI}(\sig)} = z_{1} \sum_{\mu_{1} \in \Ome_{Q}(\nu_{1})} (\mu_{1})_{C_{\CI}(\sig)} + \cdots + z_{r} \sum_{\mu_{r} \in \Ome_{Q}(\nu_{r})} (\mu_{r})_{C_{\CI}(\sig)}$$ where the $Q$-orbits $\Ome_{Q}(\nu_{1}), \ldots, \Ome_{Q}(\nu_{r})$ are all distinct. In particular, we deduce that $$z_{i} = |\set{\mu \in \Ome_{P}(\sig)}{\mu_{C_{\CI}(\sig)} = \nu_{i}}|$$ for all $1 \leq i \leq r$,  and hence $\seq{z}{r}$ are positive integers. Since $\lam_{C_{\CI}(\sig)} = (\lam_{0})_{C_{\CI}(\sig)} \in \Ome_{Q}(\nu_{i})$ for some $1 \leq i \leq r$, we conclude that $$(\vsig_{\lam})_{C_{Q}(\sig)} = m \vsig_{0} + \zet$$ where $m = |\set{\mu \in \Ome_{P}(\sig)}{\mu_{C_{\CI}(\sig)} = \lam_{C_{\CI}(\sig)}}|$ and $\map{\zet}{C_{Q}(\sig)}{\C}$ is a superclass function satisfying $\frob{\vsig_{0}}{\zet} = 0$; moreover, since $\tet \in \irr(C_{Q}(\sig))$ is a linear constituent of $\vsig_{0}$, \reft{super3} implies that $\frob{\tet}{\zet} = 0$. It follows that $$m = \frob{\tet}{\vsig_{0}} = \frob{\tet}{(\vsig_{\lam})_{C_{Q}(\sig)}} = \frob{\tet^{C_{P}(\sig)}}{\vsig_{\lam}} = \frob{\chi}{\vsig_{\lam}},$$ and hence our claim is equivalent to showing that $$|C_{\CI}(\sig)^{\perp}| = \chi(1) = m = |\set{\mu \in \Ome_{P}(\sig)}{\mu_{C_{\CI}(\sig)} = \lam_{C_{\CI}(\sig)}}|.$$

Since the mapping $\mu \mapsto \mu_{C_{\CJ}(\sig)}$ defines a bijection $\map{\pi_{\CJ}}{\set{\mu \in \jc}{\mu^{\sig} = \mu}}{\jcc}$, it also defines a bijection $\map{\pi_{\CJ}}{\set{\mu \in \lam + \CI^{\perp}}{\mu^{\sig} = \mu}}{\lam_{C_{\CI}(\sig)} + C_{\CI}(\sig)^{\perp}}$; we recall that $\CI$ is $\sig$-invariant. Since $\lam + \CI^{\perp} \sset P\lam P$ (by \refp{lin}), we have $\set{\mu \in \lam + \CI^{\perp}}{\mu^{\sig} = \mu} = \Ome_{P}(\lam) \cap (\lam + \CI^{\perp})$, and thus $$|C_{\CI}(\sig)^{\perp}| = |\Ome_{P}(\lam) \cap (\lam + \CI^{\perp})| = |\set{\mu \in \Ome_{P}(\lam)}{\mu_{\CI} = \lam_{\CI}}|.$$ On the other hand, the bijection $\map{\pi_{\CI}}{\set{\mu \in \CI^{\circ}}{\mu^{\sig} = \mu}}{C_{\CI}(\sig)^{\circ}}$ also gives $$\set{\mu \in \Ome_{P}(\lam)}{\mu_{\CI} = \lam_{\CI}} = \set{\mu \in \Ome_{P}(\lam)}{\mu_{C_{\CI}(\sig)} = \lam_{C_{\CI}(\sig)}},$$ and thus we conclude that $m = |C_{\CI}(\sig)^{\perp}| = \chi(1)$, as required. This concludes the proof of \reft{super2}.
\end{proof}

Before we close this section, we give a brief summary of the principal results we obtained so far. Given a $\sig$-invariant algebra subgroup $P = 1+\CJ$ of $\ax$, we consider the action of $P$ on the left of $C_{\CJ}(\sig)$ defined by $x\cdot a = x\inv a x^{\sig}$ for all $x \in P$ and all $a \in C_{\CJ}(\sig)$, and denot by $\Ome_{P}(a)$ the orbit which contains an element $a \in C_{\CJ}(\sig)$. Then, for every $x \in C_{P}(\sig)$ the \text{superclass} of $C_{P}(\sig)$ which contains $x$ can be defined to be the image $\VPhi(\Ome_{P}(a))$ where $\map{\VPhi}{C_{\CJ}(\sig)}{C_{P}(\sig)}$ is the Cayley transform and $a \in C_{\CJ}(\sig)$ is such that $x = \VPhi(a)$. On the other hand, $P$ also acts on the left of the dual group $\jcc$ via the contragradient action given by $x\cdot \lam = x\inv \lam x^{\sig}$ for all $x \in P$ and all $\lam \in \jcc$. For every $\lam \in \jcc$, we denote by $\Ome_{P}(\lam)$ the orbit which contains $\lam$, and define the \textit{supercharacter} $\vsig_{\lam}$ of $C_{P}(\sig)$ to be the sum $$\vsig_{\lam} = \sum_{\mu \in \Ome_{P}(\lam)} \mu\circ\VPsi$$ where $\map{\VPsi}{C_{P}(\sig)}{C_{\CJ}(\sig)}$ is the inverse of the Cayley transform. We proved that for every $\lam \in \jcc$, the function $\vsig_{\lam}$ is in fact a character of $C_{P}(\sig)$ (\reft{super2}), and that $$\vsig_{\lam} = \sig_{X(\lam)} = \sum_{\chi \in X(\lam)} \chi(1) \chi$$ where $X(\lam) = \irr_{\lam}(C_{P}(\sig))$ denotes the set of all irreducible constituents of $\vsig_{\lam}$. Also, we showed that as $\lam$ runs over a set of representatives for the orbits of $P$ on $\jcc$ the sets $X(\lam)$ partition $\irr(C_{P}(\sig))$, and that together with the partition of $C_{P}(\sig)$ into superclasses they form a \textit{supercharacter theory} for $C_{P}(\sig)$; notice that for every $\lam \in \jcc$, the supercharacter $\vsig_{\lam}$ is clearly constant on each superclass, and that the number of superclasses equals the number of supercharacters (\reft{scf1}).


\section{The classical groups} \label{sec:classical}

In this section, we illustrate our construction in the special case where $\map{\sig}{\CM_{n}(\k)}{\CM_{n}(\k)}$ is a canonical involution on $\CM_{n}(\k)$ (as defined in the introduction); we will also assume that the upper unitriangular subgroup $\UT_{n}(\k)$ of $\GL_{n}(\k)$ is $\sig$-invariant. Thus, if $G = C_{\GL_{n}(\k)}(\sig)$ denotes the $\sig$-fixed subgroup of $\GL_{n}(\k)$, then $G$ is one of the following (finite) classical groups of Lie type (defined over $\k$): the {\it symplectic group} $Sp_{2m}(\k)$, the {\it orthogonal groups} $O^{+}_{2m}(\k)$ or $O_{2m+1}(\k)$, and the {\it unitary group} $U_{n}(\k)$. (As we mentioned in the introduction, if $\sig$ is such that $G$ is the orthogonal group $O^{-}_{2m+2}(\k)$, then $\UT_{n}(\k)$ has to be replaced by its maximal algebra subgroup $\UT_{n}(\k) \cap \UT_{n}(\k)^{\sig}$; since the supercharacter theory of this subgroup has a slighty different parametrization than that of $\UT_{n}(F)$, we skip the description and leave it as an exercise for the interested reader.) Thus, throughout the section, $P$ will stand for the (upper) unitriangular group $\UT_{n}(\k)$, and we assume that the involution $\sig$ is choosen so that $P$ is $\sig$-invariant. It is straighforward to check that $C_{P}(\sig)$ consists of all (block) matrices of the form
\begin{equation} \label{eq:e1}
\begin{pmatrix} x & xu & xz \\ 0 & I_{r} & -\bar{u}^{t}J \\ 0 & 0 & J\bar{x}^{-t}J \end{pmatrix}
\end{equation}
where $J = J_{m}$ (see the introduction), $x \in \UT_{m}(\k)$, $u \in \CM_{m\x r}(\k)$ and $z \in \CM_{m}(\k)$ satisfy the relations of the following table:\medskip
{\renewcommand{\arraystretch}{1.5}
$$\begin{array}{|c|c|} \hline
\hspace{5mm}\text{Classical group}\hspace{5mm} & \hspace{2.5cm}\text{Relations}\hspace{2.5cm} \\ \hline
Sp_{2m}(\k) & r = 0,\ Jz^{t} - zJ = 0, \\ \hline
O^{+}_{2m}(\k) & r = 0,\ Jz^{t} + zJ = 0 \\ \hline
O_{2m+1}(\k) & r = 1,\ Jz^{t} + zJ = -uu^{t} \\ \hline
U_{2m}(\k) & r = 0,\ J\bar{z}^{t} + zJ = 0 \\ \hline
U_{2m+1}(\k) & r = 1,\ J\bar{z}^{t} + zJ = -u\bar{u}^{t} \\
\hline
\end{array}$$}\medskip

\hspace{-\parindent}We note that $P = 1 + \CJ$ is the algebra group which is associated with the $\sig$-invariant nilpotent uppertriangular subalgebra $\CJ = \fru_{n}(\k)$ of $\CM_{n}(\k)$, and thus $$C_{P}(\sig) = \VPhi(C_{\CJ}(\sig))$$ where $\map{\VPhi}{\CJ}{P}$ is the Cayley transform. Then, $C_{\CJ}(\sig)$ consists of all (block) matrices of the form
\begin{equation}
\begin{pmatrix} a & u & w \\ 0 & 0_{r} & -\bar{u}^{t}J \\ 0 & 0 & -J\bar{a}^{t}J \end{pmatrix}
\end{equation}
where
where $J = J_{m}$, $x \in \UT_{m}(\k)$, $u \in \CM_{m\x r}(\k)$ and $z \in \CM_{m}(\k)$ satisfy the relations of the following table:\medskip
{\renewcommand{\arraystretch}{1.5}
$$\begin{array}{|c|c|} \hline
\hspace{5mm}\text{Classical group}\hspace{5mm} & \hspace{2.5cm}\text{Relations}\hspace{2.5cm} \\ \hline
Sp_{2m}(\k) & r = 0,\ Jw^{t} - wJ = 0, \\ \hline
O^{+}_{2m}(\k) & r = 0,\ Jw^{t} + wJ = 0 \\ \hline
O_{2m+1}(\k) & r = 1,\ Jw^{t} + wJ = 0 \\ \hline
U_{2m}(\k) & r = 0,\ J\bar{w}^{t} + wJ = 0 \\ \hline
U_{2m+1}(\k & r = 1,\ J\bar{w}^{t} + wJ = 0 \\
\hline
\end{array}$$}\medskip

Superclasses and supercharacters of $P$ are parametrised by pairs $(\CD,\vphi)$ where $\CD$ is a \textit{basic subset} of $\sn = \set{(i,j)}{1 \leq i < j \leq n}$ and $\map{\vphi}{\CD}{\kx}$ is any map. By definition, a subset $\CD \sset \sn$ is said to be {\it basic} if it contains at most one entry from each row and at most one root from each column; in other words, $\CD$ is basic if $|\set{j}{i < j \leq n,\ (i,j) \in \CD}| \leq 1$ and $|\set{i}{1 \leq i < j,\ (i,j) \in \CD}| \leq 1$ for all $1 \leq i, j \leq n$. Henceforth, we will refer to such a pair $(\CD,\vphi)$ as a {\it basic pair} for $P$. For $(i,j) \in \sn$, we denote by $e_{i,j}$ the $(i,j)$th unit matrix with $1$ in the $(i,j)$th position and $0$'s elsewhere; hence, $\set{e_{i,j}}{(i,j) \in \sn}$ is the standard $\k$-basis of $\CJ$. For every basic pair $(\CD,\vphi)$, we define $$\ed = \sum_{(i,j) \in \CD} \vphi(i,j) e_{i,j} \in \CJ;$$ notice that, if $\CD$ is empty, then the sum is empty, and hence $\ed = 0$ (by convention, in this extreme case, we agree that $\vphi$ is the empty function). In virtue of \refl{scl1}, we define the \textit{superclass} $\HKD$ of $P$ to be the subset $$\HKD = \VPhi\big(P \ed P\big)$$ of $P$; notice that $\HKD$ contains the element $\VPhi(\ed) = 1+2\ed$. We have:
\begin{enumerate}
\item If $\HCK$ is a superclass of $P$, then $\HCK = \HKD$ for some basic pair $(\CD,\vphi)$.
\item If $(\CD,\vphi)$ and $(\CD',\vphi')$ are basic pairs for $P$, then $\HKD \cap \HCK_{\CD'}(\vphi') \neq \emptyset$ if and only if $(\CD,\vphi) = (\CD',\vphi')$.
\end{enumerate}

As in \refs{super1}, the superclasses of the $\sig$-fixed subgroup $C_{P}(\sig)$ are defined to be the non-empty intersections $$\KD = \HKD \cap C_{P}(\sig)$$ where $(\CD,\vphi)$ is a basic pair for $P$; moreover, by \refp{scl1}, this intersection is non-empty if and only if $\HKD$ is $\sig$-invariant. In fact, for a fixed basic pair $(\CD,\vphi)$, the action of $\sig$ defines a superclass $\HKD^{\sig}$ (by \refl{scl1}), and thus there is a basic pair $(\ds,\phs)$ such that $$\HKD^{\sig} = \HCK_{\ds}(\phs) = \VPhi\big(P e_{\ds}(\phs) P\big);$$ in particular, it follows that $\HKD$ is $\sig$-invariant if and only if $\ds = \CD$ and $\phs = \vphi$. By \refl{scl2}, we know that $$\HCK_{\ds}(\phs) = \VPhi\big(P\ed^{\sig}P\big).$$ Since $\sig$ is canonical, we have $(e_{i,j})^{\sig} = -\sig(e_{i,j}) = \pm e_{n-j+1,n-i+1}$ for all $(i,j) \in \sn$, and so $\ed^{\sig} = e_{\ds}(\phs)$. In particular, we conclude that $\HKD$ is $\sig$-invariant if and only if $\ed \in C_{\CJ}(\sig)$, and thus \reft{scl2} implies the following result. Here, and henceforth, we say that a basic pair $(\CD,\vphi)$ for $P$ is \textit{$\sig$-invariant} if $(\ds,\phs) = (\CD,\vphi)$ (hence, $(\CD,\vphi)$ is $\sig$-invariant if and only if $\ed \in C_{\CJ}(\sig)$); similarly, we say that a basic subset $\CD$ of $\sn$ is \textit{$\sig$-invariant} if $\ds = \CD$.

\begin{proposition} \label{prop:sclnt}
There is a one-to-one correspondence between superclasses of $C_{P}(\sig)$ and $\sig$-invariant basic pairs for $P$, where the superclass $\KD$ which corresponds to a $\sig$-invariant basic pair $(\CD,\vphi)$ is given by $$\KD = \set{\VPhi(x\ed x^{-\sig})}{x \in P}.$$
\end{proposition}

\begin{remark} \label{rmk:disjoint}
It is clear that every $\sig$-invariant basic subset $\CD$ of $\sn$ decomposes as a disjoint union $\CD = \CD_{1} \sqcup (\CD_{1})^{\sig} \sqcup \CD_{0}$ where
\begin{align*}
\CD_{1} &= \set{(i,j) \in \CD}{i \leq m,\ j < n-i+1},\:\text{and}\\
\CD_{0} &= \set{(i,n-i+1)}{i \leq m,\ (i,n-i+1)\in \CD};
\end{align*}
notice that $(\CD_{1})^{\sig} = \set{(n-j+1,n-i+1)}{(i,j) \in \CD_{1}}$ and that $(\CD_{0})^{\sig} = \CD_{0}$. On the other hand, if $\map{\vphi}{\CD}{\kx}$ is any map and $|\ks| = q$, then $\ed \in C_{\CJ}(\sig)$ if and only if
\begin{equation} \label{eq:vphi}
\vphi(n-j+1,n-i+1) = \begin{cases} - \vphi(i,j)^{q}, & \text{if $j \leq m+r$,} \\ -\vphi(i,j)^{q}, & \text{if $m+r < j$ and $G \neq Sp_{2m}(\k)$,} \\ \vphi(i,j), & \text{if $m < j$ and $G = Sp_{2m}(\k)$,} \end{cases}
\end{equation}
for all $(i,j) \in \CD_{1} \cup \CD_{0}$. In particular, we deduce that
\begin{itemize}
\item if, either $G = O^{+}_{2m}(\k)$, or $G = O_{2m+1}(\k)$, then $\vphi(i,n-i+1) = 0$ for all $1 \leq i \leq m$;
\item if $G = U_{n}(\k)$, then $\vphi(i,n-i+1) + \vphi(i,n-i+1)^{q} = 0$ for all $1 \leq i \leq m$.
\end{itemize}
\end{remark}

Next, we consider supercharacters, and we start by recalling the construction of the supercharacter $\hxd$ of $P$ which is associated with a given basic pair $(\CD,\vphi)$; for the details, we refer to \refs{super2}. We fix any non-trivial $\sig$-invariant linear character $\map{\tet}{\k^{+}}{\cx}$; thus, since $\sig$ acts on $\k$ as the Frobenius automorphism, we have $\tet(\alp^{q}) = \tet(\alp)$ for all $\alp \in \k$. Then, we define the linear character $\map{\ld}{\CJ^{+}}{\cx}$ of the additive group $\CJ^{+}$ by the rule $$\ld(a) = \prod_{(i,j) \in \CD} \tet(\vphi(i,j) a_{ij})$$ for all $a \in \CJ$, and let $$L_{\CD} = L(\ld) = \set{x \in P}{x\ld = \ld}$$ be the centraliser of $\ld$ with respect to the left $P$-action on $\CJ$. It is routine to check that $L_{\CD}$ consists of all matrices $x \in P$ which satisfy $x_{ik} = 0$ for all $(i,j) \in \CD$ and all $i < k < j$ (hence, $L_{\CD}$ does not depend on the map $\vphi$), and that the mapping $$x \mapsto \ld\big(\VPsi(x)\big)$$ defines a linear character $\map{\hthd}{L_{\CD}}{\cx}$. Then, we define the \textit{supercharacter} $\hxd$ of $P$ to be the induced character $$\hxd = \hthd^{P}.$$ In particular, if $\CD = \{(i,j)\}$ consists of a single entry $(i,j) \in \sn$ and $\map{\vphi}{\CD}{\kx}$ is given by $\vphi(i,j) = \alp \in \kx$, then we write $\lam_{i,j}(\alp)$, $\htet_{i,j}(\alp)$ and $\hxi_{i,j}(\alp)$ instead of $\ld$, $\hthd$ and $\hxd$, respectively; if this is the case, then we refer to the supercharacter $\hxi_{i,j}(\alp) = \htet_{i,j}(\alp)^{P}$ as the \textit{$(i,j)$th elementary character of $P$ associated with $\alp$}. In the general case, since $L_{\CD} = \bigcap_{(i,j) \in \CD} L_{i,j}$ where we write $L_{i,j} = L_{\{(i,j)\}}$, it is not difficult to prove that the supercharacter $\hxd$ factorises as the product
\begin{equation} \label{eq:hxd}
\hxd = \prod_{(i,j) \in \CD} \hxi_{i,j}(\vphi(i,j))
\end{equation}
of elementary supercharacters (see, for example, \cite[Theorem~1]{And5}); we also note that every elementary supercharacter is in fact an irreducible character of $P$ (see \cite[Lemma~2]{And3}, or \cite[Corollary~5.11]{DiaIsa}).

If $(\CD,\vphi)$ is any basic pair for $P$, then the action of $\sig$ defines a supercharacter $\hxd^{\sig}$ which corresponds to the linear character $\ld^{\sig}$ of $\CJ^{+}$ (by \refl{sch1}). Since $\tet$ is $\sig$-invariant, it is easy to check that for all $(i,j) \in \sn$ and all $\alp \in \kx$ we have $$\lam_{i,j}(\alp) = \begin{cases} \lam_{n-j+1,n-i+1}(-\alp^{q}), & \text{if $j \leq m+r$,} \\ \lam_{n-j+1,n-i+1}(-\alp^{q}), & \text{if $m+r < j$ and $G \neq Sp_{2m}(\k)$,} \\ \lam_{n-j+1,n-i+1}(\alp), & \text{if $m < j$ and $G = Sp_{2m}(\k)$,} \end{cases}$$ where $q = |\ks|$, and this clearly implies that $\ld^{\sig} = \lam_{\ds}(\phs)$ where the basic pair $(\ds,\phs)$ is as above. Therefore, we have
\begin{equation} \label{eq:xids}
\hxd^{\sig} = \hxi_{\ds}(\phs),
\end{equation}
and it follows that $\hxd$ is $\sig$-invariant if and only if the basic pair $(\CD,\vphi)$ is $\sig$-invariant. By \refp{sch2}, we obtain the following result.

\begin{proposition} \label{prop:schnt}
There is a one-to-one correspondence between supercharacters of $C_{P}(\sig)$ and $\sig$-invariant basic pairs for $P$.
\end{proposition}

In what follows, we fix an arbitrary $\sig$-invariant basic pair $(\CD,\vphi)$ for $P$, and consider the supercharacter of $C_{P}(\sig)$ which is associated with $(\CD,\vphi)$. On the one hand, let $$\Ome_{\CD}(\vphi) = \set{x\inv\ld x^{\sig}}{x \in P}$$ be the subset of $P\ld P$ consisting of $\sig$-fixed elements (see \refp{fixed}), and define the map $\map{\sd}{C_{P}(\sig)}{\C}$ by the rule
\begin{equation} \label{eq:sd}
\sd(x) = \sum_{\lam \in \Ome_{\CD}(\vphi)} \mu(\VPsi(x))
\end{equation}
for all $x \in C_{P}(\sig)$. By \reft{super2}, $\sd$ is a character of $C_{P}(\sig)$, and in fact $$\sd = \sum_{\chi \in X_{\CD}(\vphi)} \chi(1)\, \chi$$ where $X_{\CD}(\vphi) = X(\ld)$ denotes the set  consisting of all irreducible constituents of $\sd$; we recall that $X_{\CD}(\vphi)$ can also be described as the set consisting of all irreducible characters $\chi \in \irr(C_{P}(\sig))$ such that the Glauberman correspondent $\hchi \in \irr_{\sig}(P)$ of $\chi$ is a constituent of the supercharacter $\hxd$ of $P$. The results of \refs{super2} imply the following.

\begin{theorem} \label{thm:schth}
If\; $\D$ denotes the set of all $\sig$-invariant basic pairs for $P$, then the sets $\CX = \set{\sd}{(\CD,\vphi) \in \D}$ and $\CY = \set{\KD}{(\CD,\vphi) \in \D}$ form a supercharacter theory for $C_{P}(\sig)$.
\end{theorem}

Although the supercharacters are defined is a different way,  in the case of the symplectic and orthogonal groups this supercharacter theory for $C_{P}(\sig)$ turns out to be the same as the one described in the papers \cite{AndNet1, AndNet2, AndNet3}; in fact, \cite[Theorem~6.1]{AndNet3} asserts that, up to the multiplication by a positive integer, the supercharacter $\sd$ can be obtained by inducting a linear character of a suitable subgroup of $C_{P}(\sig)$. To see this, we first define the subgroup $Q_{\CD}$ of $P$ as follows: for every $(i,j) \in \sn$ let
\begin{align*}
Q_{i,j} &= L_{i,j},\quad \text{if $j \leq m$,}\\
Q_{i,j} &= \set{x \in P}{x_{i,k} = x_{k,j} = 0 \text{ for all } i < k \leq m},\quad \text{if $i \leq m < j$;}\\
Q_{i,j} &= (L_{n-j+1,n-i+1})^{\sig},\quad \text{if $m < i$;}
\end{align*}
then, $$Q_{\CD} = \bigcap_{(i,j) \in \CD} Q_{i,j}.$$ On the other hand, for every map $\map{\vphi}{\CD}{\kx}$, we define $\map{\htd}{Q_{\CD}}{\cx}$ by $$\htd(x) = \ld\big(\VPsi(x)\big)$$ for all $x \in Q_{\CD}$. It is easy to check that $\htd$ is a linear character of $Q_{\CD}$; moreover, by \cite[Lemma~2.1]{AndNet2} it follows that
\begin{equation} \label{eq:basic2}
\hxd = \htd^{P}
\end{equation}
(see also \refp{sch0}). If the basic pair $(\CD,\vphi)$ is $\sig$-invariant, then it is straightforward to check that the subgroup $Q_{\CD}$ and the linear character $\htd$ are both $\sig$-invariant; if this is the case, we denote by $\td$ the restriction of $\htd$ to the $\sig$-fixed subgroup $C_{Q_{\CD}}(\sig)$, and define
\begin{equation} \label{eq:xd}
\xd = \td^{C_{P}(\sig)}.
\end{equation}
We claim that there exists a positive integer $n_{\CD,\vphi}$ such that $\sd = n_{\CD,\phi} \xd$. To see this, we first prove the following general result (which extends \refl{sch2}).

\begin{proposition} \label{prop:basic1}
Let $\CJ$ be a $\sig$-invariant nilpotent subalgebra of $\CA$, and let $P = 1+\CJ$. Let $\CI$ be a $\sig$-invariant subalgebra of $\CJ$, let $Q = 1+\CI$, let $\htet \in \irr_{\sig}(Q)$ and let $\tet = \pi_{Q}(\htet) \in \irr(C_{Q}(\sig))$ be the Glauberman correspondent of $\htet$. Moreover, let $\chi \in \irr(C_{P}(\sig))$, and let $\hchi \in \irr_{\sig}(P)$ be such that $\pi_{P}(\hchi) = \chi$. Then, $\frob{\chi}{\tet^{C_{P}(\sig)}} \neq 0$ if and only if $\frob{\hchi}{\htet^{P}} \neq 0$.
\end{proposition}

\begin{proof}
We proceed by induction on $\dim \CJ$. Firstly, suppose that $\CI + \CJ^{2} = \CJ$. Then, by \cite[Lemma~3.1]{Isa2}, we have $\CI = \CJ$; hence, $Q = P$ and there is nothing to prove. Otherwise, let $N = 1+(\CI+\CJ^{2})$; hence, $Q \sset N \subsetneq P$. Then, since $\CJ^{2} \sset \CI+\CJ^{2}$ and since $\CJ^{2}$ is clearly $\sig$-invariant, $N$ is a $\sig$-invariant normal subgroup of $P$. Now, let us assume that $\frob{\chi}{\tet^{C_{P}(\sig)}} \neq 0$. Then, by Frobenius reciprocity, we have $\frob{\chi}{(\tet^{C_{N}(\sig)})^{C_{P}(\sig)}} = \frob{\chi_{C_{N}(\sig)}}{\tet^{C_{N}(\sig)}}$, and thus there exists $\tau \in \irr(C_{N}(\sig))$ such that $\frob{\tau}{\chi_{C_{N}(\sig)}} \neq 0$ and $\frob{\tau}{\tet^{C_{N}(\sig)}} \neq 0$. Since $\frob{\tau}{\chi_{C_{N}(\sig)}} = \frob{\tau^{C_{P}(\sig)}}{\chi}$, \cite[Theorem~(13.29)]{Isa1} implies that $\frob{\htau^{P}}{\hchi} \neq 0$ where $\htau \in \irr_{\sig}(N)$ is such that $\pi_{N}(\htau) = \tau$. On the other hand, by induction, we also have $\frob{\htau}{\htet^{N}} \neq 0$, and thus $\htau^{P}$ is a constituent of $\htet^{P} = (\htet^{N})^{P}$. Since $\hchi$ is a constituent of $\htau^{P}$, we conclude that $\frob{\hchi}{\htet^{P}} \neq 0$, as required. Conversely, suppose that $\frob{\hchi}{\htet^{P}} \neq 0$; thus, $\frob{\hchi_{N}}{\htet^{N}} \neq 0$ (by Frobenius reciprocity). By \cite[Theorem~(13.27)]{Isa1}, there exists $\htau \in \irr_{\sig}(N)$ such that $\frob{\htau}{\hchi_{N}} \neq 0$. Then, $\frob{\htau}{\htet^{N}} \neq 0$, and so by induction we obtain $\frob{\tau}{\tet^{C_{N}(\sig)}} \neq 0$ where $\tau = \pi_{N}(\htau) \in \irr(C_{N}(\sig))$. Since $\frob{\htau}{\hchi_{N}} = \frob{\htau^{P}}{\hchi}$, \cite[Theorem~(13.29)]{Isa1} implies that $\frob{\tau^{C_{P}(\sig)}}{\chi} \neq 0$. Since $\tau^{C_{P}(\sig)}$ is a constituent of $\tet^{C_{P}(\sig)} = (\tet^{C_{N}(\sig)})^{C_{P}(\sig)}$, we conclude that $\frob{\chi}{\tet^{C_{P}(\sig)}} \neq 0$, and this completes the proof.
\end{proof}

We are now able to prove the following result.

\begin{lemma} \label{lem:basic1}
If $(\CD,\vphi)$ be a $\sig$-invariant basic pair for $P$, then the characters $\xd$ and $\sd$ of $C_{P}(\sig)$ have the same irreducible constituents. In particular, if $(\CD,\vphi)$ and $(\CD',\vphi')$ are $\sig$-invariant basic pairs for $P$, then $\frob{\xd}{\xi_{\CD'}(\vphi')} \neq 0$ if and only if $(\CD,\vphi) = (\CD',\vphi')$.
\end{lemma}

\begin{proof}
If $\chi \in \irr(C_{P}(\sig))$ is an irreducible constituent of $\xd = \td^{C_{P}(\sig)}$, then the previous proposition asserts that the Glauberman correspondent $\hchi \in \irr_{\sig}(P)$ of $\chi$ is a constituent of $\hxd = \htd^{C_{P}(\sig)},$ and thus $\chi$ is an irreducible constituent of $\sd$ (by \reft{super2}). Conversely, if $\chi \in \irr(C_{P}(\sig))$ is an irreducible constituent of $\sd$, then $\hchi$ is an irreducible constituent of $\hxd$, and so $\chi$ is an irreducible constituent of $\xd$ (by the previous proposition). For the last assertion, it is enough to recall that $\frob{\sd}{\vsig_{\CD'}(\vphi')} \neq 0$ if and only if $(\CD,\vphi) = (\CD',\vphi')$.
\end{proof}

We next show that $\xd$ is a superclass function on $C_{P}(\sig)$. Since the basic subset $\CD \sset \sn$ is $\sig$-invariant, we have a decomposition $\CD = \CD_{1} \sqcup (\CD_{1})^{\sig} \sqcup \CD_{0}$ where $\CD_{1}$ and $\CD_{0}$ are as in \refr{disjoint}. On the other hand, since the basic pair $(\CD,\vphi)$ is $\sig$-invariant, \refeqs{hxd}{xids} imply that the supercharacter $\hxd$ factorises as the product $$\hxd = \hxi_{\CD_{1}}(\vphi_{1})\, \hxi_{\CD_{1}}(\vphi_{1})^{\sig}\, \hxi_{\CD_{0}}(\vphi_{0})$$ where $\vphi_{1}$ and $\vphi_{0}$ denote the restriction of $\vphi$ to $\CD_{1}$ and $\CD_{0}$, respectively. Since $\hxi_{\CD_{1}}(\vphi_{1})$ and $\hxi_{\CD_{1}}(\vphi_{1})^{\sig}$ have the same restriction to $C_{P}(\sig)$, we conclude that
\begin{align*}
\hxd_{C_{P}(\sig)} &= \big( \hxi_{\CD_{1}}(\vphi_{1})_{C_{P}(\sig)} \big)^{2}\cdot \hxi_{\CD_{0}}(\vphi_{0})_{C_{P}(\sig)} \\ &= \prod_{(i,j) \in \CD_{1}} \big( \hxi_{i,j}(\vphi(i,j))_{C_{P}(\sig)} \big)^{2} \cdot \prod_{(i,j) \in \CD_{0}} \hxi_{i,j}(\vphi(i,j))_{C_{P}(\sig)}.
\end{align*}

\begin{remark}
We observe that, for all $(i,j) \in \sn$ and all $\alp \in \kx$, the square power $\hxi_{i,j}(\alp)^{2}$ is a superclass function of $P$, and thus it decomposes as a linear combination of supercharacters (with integer coefficients); furthermore, from \cite[Lemma~11]{And1} (see also \refp{sch0}) it follows that $\hxi_{i,j}(2\alp)$ is an irreducible constituent of $\hxi_{i,j}(\alp)^{2}$ with multiplicity equal to $1+(q-1)(j-i+1)$ where $q = |\ks|$.
\end{remark}

Henceforth, for every $(i,j) \in \sn$ with $j \leq n-i+1$ and every $\alp \in \kx$, we will simplify the notation and write $\xi_{i,j}(\alp)$ (resp., $\vsig_{i,j}(\alp)$) to denote the character $\xd$ (resp., the supercharacter $\sd$) of $C_{P}(\sig)$ where $(\CD,\vphi)$ is the $\sig$-invariant basic pair with $\CD = \{(i,j),(n-j+1,n-i+1)\}$ and $\alp = \vphi(i,j)$; as before, we refer to $\xi_{i,j}(\alp)$ as the \textit{$(i,j)$th elementary character of $C_{P}(\sig)$ associated with $\alp$}. Similarly to the case of the unitriangular group, we have the following factorisation; for a proof, see \cite[Proposition~3]{AndNet1}.

\begin{theorem} \label{thm:basic2}
If $(\CD,\vphi)$ is a $\sig$-invariant basic pair for $P$, then $$\xd = \prod_{(i,j) \in \CD'} \xi_{i,j}(\vphi(i,j))$$ where $\CD' = \set{(i,j) \in \CD}{j \leq n-i+1}$.
\end{theorem}

In view of this theorem, the goal of proving that the $\xd$ is a superclass function of $C_{P}(\sig)$ reduces to proving that this holds for every elementary character.

\begin{lemma} \label{lem:elem1}
Let $(i,j) \in \sn$ be such that $< j \leq n-i+1$, and let $\alp \in \kx$. Then, $\xi_{i,j}(\alp) = \hxi_{i,j}(2\alp)_{C_{P}(\sig)}$, and hence $\xi_{i,j}(\alp)$ is a superclass function on $C_{P}(\sig)$. In particular, there exists a constant $n_{i,j}(\alp)$ such that $\xi_{i,j}(\alp) = n_{i,j}(\alp) \vsig_{i,j}(\alp)$.
\end{lemma}

\begin{proof}
For simplicity, we set $\hxi = \hxi_{i,j}(2\alp)$; as for \refeq{basic2}, \cite[Lemma~2.1]{AndNet2} implies that $\hxi = \htau^{P}$ where $\htau = \htau_{i,j}(\alp)$ is the linear character of $Q = Q_{i,j}$ defined by $$\htau(x) = \tet(2\alp x_{i,j})$$ for all $x \in Q$. Since $P = QC_{P}(\sig)$, we obtain $$\hxi_{C_{P}(\sig)} = \big( \htau_{Q \cap C_{P}(\sig)} \big)^{C_{P}(\sig)}= \big( \htau_{C_{Q}(\sig)} \big)^{C_{P}(\sig)}$$ (by Mackey's criterion; see \cite[Exercise~6.1]{Isa1}). Since $\htau(x) = \tet(2\alp x_{i,j}) = \tet(\alp x_{i,j})^{2} = \tau_{\CD}(\vphi)(x)$ for all $x \in C_{Q}(\sig)$, we conclude that $\hxi_{C_{P}(\sig)} = \xi_{i,j}(\alp)$, and thus $\xi_{i,j}(\alp)$ is a superclass fuction on $C_{P}(\sig)$ (because $\hxi$ is a superclass function on $P$). It follows that $\xi_{i,j}(\alp)$ is a linear combination of the supercharacters of $C_{P}(\sig)$, and hence $\xi_{i,j}(\alp)$ must be a multiple of $\vsig_{i,j}(\alp)$ (by \refl{basic1}).
\end{proof}

On the other hand, we consider the restriction to $C_{P}(\sig)$ of a $\sig$-invariant elementary character $\hxi_{i,n-i+1}(\alp)$ where $i \leq m$ and $\alp \in \kx$; the assumption of being $\sig$-invariant implies that, either $G = Sp_{2m}(\k)$, or $G = U_{n}(\k)$ and $\alp \in \k$ satisfies $\alp^{q} = -\alp$ where $q = |\ks|$. Since $\hxi_{i,n-i+1}(\alp)$ is an irreducible character of $P$ (\cite[Lemma~2]{And3}, or \cite[Corollary~5.11]{DiaIsa}), Glauberman's Theorem guarantees that its restriction to $C_{P}(\sig)$ has a unique irreducible constituent with odd multiplicity, and this clearly implies that there exists a positive integer $m$ such that $\xi_{i,n-i+1}(\alp) = m \chi$ where $\chi = \pi_{P}\big( \hxi_{i,n-i+1}(\alp) \big)$. In fact, we have the following.

\begin{lemma} \label{lem:elem2}
If $i < m$ and $\alp \in \kx$ are as above, then $\xi_{i,n-i+1}(\alp)$ is an irreducible character of $C_{P}(\sig)$, and $\vsig_{i,n-i+1}(\alp) = q^{m-i+1} \xi_{i,n-i+1}(\alp)$ where $q = |\ks|$.
\end{lemma}

\begin{proof}
For simplicity, we set $\xi = \xi_{i,n-i+1}(\alp)$ and $\tau = \tau_{i,n-i+1}(\alp)$; hence, $\tau$ is a linear character of $Q = Q_{i,n-i+1}$ and $\xi = \tau^{C_{P}(\sig)}$. We observe that the group $C_{P}(\sig)$ factorises as the semidirect product $$C_{P}(\sig) = P_{0} \ltimes N$$ where $P_{0}$ is a subgroup (naturally) isomorphic to the unitriangular group $\UT_{m}(q)$ and $N$ is a normal subgroup of nilpotency class less than or equal $2$; referring to \refeq{e1}, $P_{0}$ consists of all (block) matrices with $u = 0$ and $z = 0$, and $N$ consists of all matrices with $x = I_{m}$. It is routine to check that $Q$ equals the inertia group $I_{P}(\tau_{N})$ in $P$ of the restriction $\tau_{N}$ of $\tau$ to $N$; in other words, this means tat $Q = \set{x \in P}{\tau(xyx^{-1}) = \tau(y) \text{ for all } y \in N}$. By Clifford's theorem (see \cite[Theorem~6.11]{Isa2}), we conclude that $\xi = \tau^{P}$ is an irreducible character. By the above, this implies that $\xi = \pi_{P}\big(\hxi_{i,n-i+1}(\alp)\big)$, and thus $\vsig_{i,n-i+1}(\alp) = \xi(1)\, \xi$ (by \reft{super2}). The result follows because $\xi(1) = |C_{P}(\sig):C_{Q}(\sig)| = q^{m-i+1}$.
\end{proof}

Finally, we deduce the following (required) result.

\begin{proposition} \label{prop:basic2}
If $(\CD,\vphi)$ is a $\sig$-invariant basic pair for $P$, then $\xd$ is a superclass function of $C_{P}(\sig)$, and hence there exists a constant $n_{\CD,\vphi}$ such that $\xd = n_{\CD,\vphi} \sd$.
\end{proposition}

\begin{proof}
By \reft{basic2} and by the two previous lemmas, it follows that $\xd$ is in fact a superclass function. Since supercharacters form a basis of the complex vector space consisting of all superclass functions (because they are orthogonal and in the same number as superclasses), we conclude that $\xd$ is a linear combination of supercharacters, and \refl{basic1} implies that $\xd$ must a multiple of $\sd$.
\end{proof}

As a consequence, we obtain the following result (see \cite{AndNet3} for the symplectic and orthogonal cases).

\begin{theorem} \label{thm:schth2}
If\; $\D$ denotes the set of all $\sig$-invariant basic pairs for $P$, then the sets $\CX' = \set{\xd}{(\CD,\vphi) \in \D}$ and $\CY = \set{\KD}{(\CD,\vphi) \in \D}$ form a supercharacter theory for $C_{P}(\sig)$.
\end{theorem}


\end{document}